\documentclass[11pt,reqno]{amsart}

\usepackage{times}
\usepackage{comment}
\usepackage{amsmath}
\usepackage{dsfont,mathtools,amssymb}
\usepackage{color}
\usepackage[hidelinks]{hyperref}
\mathtoolsset{showonlyrefs,showmanualtags}

\newtheorem{theorem}{Theorem}[section]
\newtheorem{proposition}[theorem]{Proposition}
\newtheorem{lemma}[theorem]{Lemma}

\newtheorem{remark}[theorem]{Remark}
\newtheorem{definition}[theorem]{Definition}
\newtheorem{example}[theorem]{Example}

\numberwithin{equation}{section}
\numberwithin{figure}{section}

\usepackage{tikz}
\usepackage{subfigure}
\usepackage{amsmath}

\definecolor{darkblue}{rgb}{0.1,0.1,0.7}

\definecolor{darkred}{rgb}{0.7,0.1,0.1}

\newcommand{\ind}{\mathbf{1}}

\newcommand{\e}{\varepsilon}

\newcommand{\be}{\begin{equation}}

\newcommand{\cA}{\ensuremath{\mathcal A}} 
 
\newcommand{\cC}{\ensuremath{\mathcal C}} 
\newcommand{\cD}{\ensuremath{\mathcal D}} 
\newcommand{\cE}{\ensuremath{\mathcal E}} 
\newcommand{\cF}{\ensuremath{\mathcal F}} 
\newcommand{\cG}{\ensuremath{\mathcal G}}

\newcommand{\cN}{\ensuremath{\mathcal N}} 
 
\newcommand{\cP}{\ensuremath{\mathcal P}} 
\newcommand{\cQ}{\ensuremath{\mathcal Q}}

\newcommand{\cT}{\ensuremath{\mathcal T}} 
\newcommand{\cU}{\ensuremath{\mathcal U}} 
\newcommand{\cV}{\ensuremath{\mathcal V}}

\newcommand{\bbE}{{\ensuremath{\mathbb E}} }

\newcommand{\bbN}{{\ensuremath{\mathbb N}} } 
 
\newcommand{\bbP}{{\ensuremath{\mathbb P}} } 
 
\newcommand{\bbR}{{\ensuremath{\mathbb R}} }

\newcommand{\si}{\sigma}

\newcommand{\tc}{\, |\, }

\newcommand{\kernel}[4]{{\cQ}(#1,#2\,;\,#3,#4)}
\newcommand{\kernelJ}[4]{{\cQ_{\bf J}}(#1,#2\,;\,#3,#4)}
\newcommand{\kernelJx}[4]{{\cQ_{{\bf J},x}}(#1,#2\,;\,#3,#4)}
\newcommand{\kernelJhat}[4]{{\widehat\cQ_{{\bf J}}}(#1,#2\,;\,#3,#4)}

\newcommand{\ignore}[1]{}

\let\a=\alpha \let\b=\beta   \let\d=\delta  \let\e=\varepsilon
 \let\g=\gamma     \let\k=\kappa  \let\l=\lambda
      \let\o=\omega      
\let\r=\rho   \let\t=\tau   
  \let\z=\zeta
     \let\L=\Lambda 
\let\O=\Omega

\def\({\left(}
\def\){\right)}
\allowdisplaybreaks

\title[Nonlinear dynamics for the Ising model]{Nonlinear dynamics for the Ising model}


\author{Pietro Caputo and Alistair Sinclair}
\address{Pietro Caputo\\ Universit\`{a} Roma Tre}
\email{pietro.caputo@uniroma3.it}
\address{Alistair Sinclair\\ University of California, Berkeley}
\email{sinclair@cs.berkeley.edu}
\thanks{PC was supported in part by the Miller Institute at UC Berkeley.  AS was supported
in part by NSF grants CCF-1815328 and CCF-2231095.
Part of this work was done while AS was visiting EPFL, Lausanne and while PC was visiting 
UC Berkeley} 

\begin{document}

\begin{abstract} 
We introduce and analyze a natural class of nonlinear dynamics for spin systems such as the Ising model.
This class of dynamics is based on the framework of mass action kinetics, which models the evolution
of systems of entities under pairwise interactions, and captures a number of important 
nonlinear models from various fields, including chemical reaction networks, Boltzmann's model of an ideal gas,
recombination in  population genetics and
genetic algorithms.  In the context of spin systems, it is a natural generalization of linear dynamics based
on Markov chains, such as Glauber dynamics and block dynamics, which are by now well understood.  
However, the inherent nonlinearity makes the dynamics much harder to analyze, and rigorous
quantitative results so far are limited to processes which converge to essentially trivial stationary
distributions that are product measures.

In this paper we provide the first quantitative convergence analysis for natural nonlinear dynamics
in a combinatorial setting where the stationary distribution contains non-trivial correlations, namely
spin systems at high temperatures.  We prove that nonlinear versions of both the Glauber dynamics
and the block dynamics converge to the Gibbs distribution of the Ising model (with given external fields)
in times $O(n\log n)$ and $O(\log n)$ respectively, where $n$ is the size of the underlying graph
(number of spins).  Given the lack of general analytical methods for such nonlinear systems, our analysis
is unconventional, and combines tools such as information percolation (due in the linear setting to Lubetzky
and Sly), a novel coupling of the Ising model with Erd\H{o}s-R\'enyi random graphs, and non-traditional 
branching processes augmented by a ``fragmentation" process.  Our results extend immediately to 
any spin system with a finite number of spins and bounded interactions.
\end{abstract}

\maketitle
\thispagestyle{empty}

\section{Introduction}\label{sec:intro}
{\it Mass action kinetics\/} is a general framework for studying systems of interacting
entities.  The framework emerged in the study of chemical reaction networks, dating
back at least to the seminal work of Horn and Jackson in the 1970s~\cite{HornJackson},
and has seen a resurgence of activity in recent years; see the monograph~\cite{FeinbergBook}.
However, it also captures a wide range of processes that are of interest in other fields,
including Boltzmann's model of an ideal gas~\cite{Boltz}, classical models of population
genetics~\cite{Hardy,Weinberg}, genetic algorithms in combinatorial optimization~\cite{goldberg,Mitchell}, 
and random sampling~\cite{Sinetal,SV13}.

We describe mass action kinetics in the special case where all interactions are pairwise and
homogeneous; this captures most of the complexity of general systems while keeping notation
and technicalities to a minimum.  Let $\O$ denote a finite set of {\it types}.
A (quadratic) mass action system is described by a directed graph whose vertices are
ordered pairs of types $(\sigma,\sigma')$, and a directed edge from $(\sigma,\sigma')$ to $(\tau,\tau')$
indicates the presence of a {\it reaction\/} in which types $\sigma,\sigma'$ 
combine to produce types~$\tau,\tau'$.  Reactions involving a specific pair $(\sigma,\sigma')$ are
governed by a {\it collision kernel\/} $\kernel{\sigma}{\sigma'}{\cdot\,}{\cdot\,}$, where $\kernel{\sigma}{\sigma'}{\tau}{\tau'}$
is the probability that the outcome of the reaction is the pair~$(\tau,\tau')$.  
We assume throughout 
the symmetry property $\kernel{\sigma}{\sigma'}{\tau}{\tau'} = \kernel{\sigma'}{\sigma}{\tau'}{\tau}$.

The state of the system at any time~$t$ is fully described by the vector $p_t$,
where $p_t(\sigma)$ is the mass of type~$\sigma$ at time~$t$, normalized so that $\sum_{\sigma\in\O} p_t(\sigma)=1$
(i.e., the $p_t(\sigma)$ can be viewed as concentrations, or probabilities).  The initial state is denoted~$p_0$.
According to the so-called
``mass action" principle, each reaction $(\sigma,\sigma')\to (\tau,\tau')$ takes place at a 
rate determined by the {\it product\/} of the current masses of types $\sigma,\sigma'$.
The dynamics of the system is described by the following set of
equations\footnote{We give the dynamics in discrete time here; a continuous
time version, in which the $\kernel{\sigma}{\sigma'}{\cdot\,}{\cdot\,}$ are reaction {\it rates}, 
can be defined analogously in the obvious way.},
one for each type~$\tau\in\O$:
\begin{equation}\label{eq:massaction}
   p_{t+1}(\tau) = \sum_{\sigma,\sigma'\!,\tau'} p_t(\sigma)p_t(\sigma') \kernel{\sigma}{\sigma'}{\tau}{\tau'}.
\end{equation}

At this level of generality such systems can be arbitrarily badly behaved (e.g., chaotic),
so it is necessary to impose standard regularity conditions.  A mass action system
is said to be {\it reversible\/}\footnote{In the mass action kinetics literature, the term
``reversible" has unfortunately been used to denote the weaker property that 
$\kernel{\sigma}{\sigma'}{\tau}{\tau'}\ne 0$ iff $\kernel{\tau}{\tau'}{\sigma}{\sigma'}\ne 0$, whereas in physics reversibility
is synonymous with detailed balance.  In this paper, we shall use the terms ``detailed balanced" and
``reversible" interchangeably to denote the stronger condition~\eqref{eq:detbal}.}
or {\it detailed balanced\/} if there exists a strictly positive mass vector $\mu=(\mu(\sigma))>0$
such that 
\begin{equation}\label{eq:detbal}
    \mu(\sigma) \mu(\sigma') \kernel{\sigma}{\sigma'}{\tau}{\tau'} = \mu(\tau) \mu(\tau') \kernel{\tau}{\tau'}{\sigma}{\sigma'}\,, \qquad \forall\, \sigma,\sigma',\tau,\tau'.
\end{equation}    
It is easy to check that any such~$\mu$ is necessarily an {\it equilibrium\/} or {\it stationary\/} point for 
the dynamics~\eqref{eq:massaction}.
A mass action system may have many positive equilibrium points, but it is known (see, e.g.,~\cite{FeinbergBook},
and also Proposition~\ref{prop:invrev} in this paper)
that if any one of them satisfies the detailed balance condition then they all do.
We stress that we do {\it not\/} require the kernel~$\cQ$ to be irreducible (i.e., the directed
graph describing it need not be strongly connected).

The mass action system defined in~\eqref{eq:massaction} above can be viewed as a natural
nonlinear analog of a reversible Markov chain, whose dynamics takes the form 
$p_{t+1}(\tau) = \sum_{\sigma} p_t(\sigma) {\mathcal Q}(\sigma\,;\,\tau)$, where now ${\mathcal Q}(\sigma\,;\,\tau)$ is the 
transition matrix of the chain and the reversibility condition is $\mu(\sigma) {\mathcal Q}(\sigma\,;\,\tau) = \mu(\tau) {\mathcal Q}(\tau\,;\,\sigma)$
for all $\sigma,\tau$.  In the linear setting, there are well known criteria for convergence to stationarity and 
there is by now a vast literature on mixing times of reversible Markov chains
and their algorithmic applications to sampling, approximate counting and integration, statistical physics, etc.
By contrast, in the nonlinear setting even the most basic questions are still open: for example, the Global
Attractor Conjecture~\cite{FeinbergBook} asserts that any detailed
balanced\footnote{Actually, this property is conjectured to hold under
the weaker condition known as ``complex balance"~\cite{FeinbergBook}; see also Section~\ref{sec:cgce}.} 
mass action system converges to a stationary point 
when started from any 
initial point~$p_0$ with full support.  (In fact many stationary points
may exist, but only one is consistent with any given initial condition~$p_0$.)  And even in particular cases 
of interest where convergence has been proved, almost nothing is known about the {\it rate\/} of convergence
(the analog of the {\it mixing time\/} for Markov chains).

In the discrete combinatorial setting, one of very few examples for which useful bounds on the convergence
rate are known is the classical Hardy-Weinberg model of genetic recombination~\cite{Hardy,Weinberg}.  
Here the types are bit strings $\O=\{0,1\}^n$ (each bit representing
an allele on a chromosome), and a reaction between two strings $\sigma,\sigma'$ involves picking a ``crossover"
subset $\Lambda\subseteq  \{1,\ldots,n\}$ of positions according to some probabilistic rule and exchanging the
bits in~$\Lambda$ between~$\sigma$ and~$\sigma'$ to obtain two new strings $\tau,\tau'$.
(For example, one classical rule is to pick~$i\in \{0,\ldots,n\}$ u.a.r.\ and let $\Lambda$ consist of the first~$i$ bits.)
It is well known that this dynamics converges to the distribution~$\mu$ in which all bits are {\it independent}, with the marginal
probabilities of a~1 at each position given by those in the initial distribution~$p_0$.  In~\cite{Sinetal2,CapSin}, the rate
of convergence was related precisely to the rate at which the strings are {\it fragmented\/} by the repeated
random cuts~$\Lambda$, thus enabling very precise estimates of the convergence time for any choice of crossover rule.

The above analysis relies crucially on the fact that in the equilibrium distribution all bits are independent.
When there is even a small amount of correlation, there appear to be no techniques available
to obtain useful bounds on convergence rates.  In this paper, we address this question for arguably the most
natural example in which correlations arise, namely the Ising model of statistical physics.
Here the types are spin configurations $\sigma\in\O=\{\pm 1\}^{V}$ which assign one of two possible spin
values $\pm 1$ to each vertex of a graph $G=(V,E)$.  The {\it Gibbs distribution\/} is given by 
\begin{equation}\label{eq:Gibbs}
   \mu_{{\bf J},{\bf h}}(\sigma) =\frac{1}{Z_{\bf J,h}} \exp\Biggl\{\frac{1}{2} \sum_{x,y \in V} J_{xy}\sigma_x\sigma_y + \sum_{x\in V} h_x\sigma_x\Biggr\},
\end{equation}
where ${\bf h} = \{h_x\}_{x\in V}$ is a vector of real numbers whose entry~$h_x$ represents the {\it external field\/} at vertex~$x$,
and ${\bf J} = \{J_{xy}\}_{x,y\in V}$ is a symmetric real matrix whose entry~$J_{xy}$ represents the {\it interaction\/}
between spins at adjacent vertices~$x,y$.  (When there is no edge between $x$ and~$y$, $J_{xy} = 0$.)  The
normalizing factor $Z_{\bf J,h}$ is the {\it partition function}.  Note that we
allow the interactions~$J_{xy}$ to be either positive (ferromagnetic) or negative (antiferromagnetic), and the fields~$h_x$ 
to be either positive (favoring $+1$ spins) or negative (favoring $-1$ spins).  
Setting ${\bf J}=\beta A$, where $\beta>0$ and $A$ is the
adjacency matrix of~$G$, corresponds to the standard ferromagnetic Ising model on~$G$ at inverse temperature~$\beta$.
We emphasize that, while for simplicity we develop our results for the specific case of the Ising model, they
hold equally for any spin system with a constant number of different spins and bounded pairwise interactions (such
as the $q$-state Potts model); see Section~\ref{sec:open} for more detail.

The classical {\it Glauber dynamics\/} for the Ising model picks a random vertex $x\in V$ at each step and resamples the
spin at~$x$ according to the correct conditional distribution given its neighboring spins; this Markov chain converges
to the Gibbs distribution~\eqref{eq:Gibbs} from any initial configuration.  The analogous nonlinear mass action kinetics
is defined by equation~\eqref{eq:massaction} with the following kernel:  Given two configurations $\sigma,\sigma'$, pick a 
random vertex~$x$ and exchange the spins $\sigma_x,\sigma'_x$, obtaining two new configurations $\tau,\tau'$.
The transition probabilities $\kernel{\sigma}{\sigma'}{\tau}{\tau'}$ are chosen to satisfy the detailed
balance condition~\eqref{eq:detbal}, where $\mu= \mu_{{\bf J},{\bf h}}$ is the Gibbs distribution.
We emphasize that, in contrast to Glauber dynamics, here the system is evolving {\it endogenously\/} via pairwise interactions between
configurations, rather than via exogenously applied spin updates.
Our first result shows that this dynamics converges to the Gibbs distribution~\eqref{eq:Gibbs}, where the 
fields~$\bf h$ are determined by the marginal probabilities of the spins at each vertex
in the initial distribution.  
(The fact that the marginals determine a unique vector of fields~$\bf h$ follows from standard convexity arguments;
see, e.g., \cite{ChenEldan22}.)
\begin{theorem}\label{thm:convintro}
Let $p_t$ denote the distribution at time~$t$ for the above mass action kinetics for the Ising model with interactions
${\bf J}$ starting from any initial distribution~$p_0$, and let ${\bf h}$ be the unique choice of external fields such
that the marginal spin probabilities at each vertex $x\in V$ in $\mu_{{\bf J},{\bf h}}$ are the same as those in~$p_0$.  
Then $p_t$ converges to $\mu_{{\bf J},{\bf h}}$ 
as $t\to\infty$.
\end{theorem}
Note that, unlike the standard Glauber dynamics, the nonlinear dynamics has conserved quantities---namely, the marginal probabilities
of the spins at each vertex---and the values of these conserved quantities determine which of the family of
stationary points the dynamics converges to.  This phenomenon is typical in mass action kinetics.
The key to our proof of Theorem~\ref{thm:convintro}, which is based on an earlier argument in~\cite{Sinetal} (see
Remark~\ref{rem:qds}), is establishing an {\it irreducibility\/} property, namely that along any trajectory, 
the probability of any configuration eventually remains uniformly bounded away from zero.

We pause to briefly mention some features of mass action kinetics that make its analysis much more complex
than that of Glauber dynamics, and which explain the lack of quantitative convergence results.  
First, as noted above, there are in general multiple equilibrium points, which are characterized by conserved quantities. 
Second, in contrast to the linear case, the total variation distance to stationarity is {\it not\/} monotonically decreasing (see \cite[Remark 2.7]{CLL})
and there are no simple coupling arguments to rely upon. 
Finally, the nonlinearity means that we do not have at our disposal a spectral theory and other functional analysis tools
that have proved so powerful in the analysis of Markov chains.
As usual in kinetic theory, a natural way to study convergence to stationarity here is to use relative entropy,  
which provides a monotonically decreasing functional; however, quantitative analysis of this quantity is a notoriously
difficult problem in the nonlinear setting, which has so far been solved only in the non-interacting case
(genetic recombination)~\cite{CapSin,CapPar}.     

Our main result establishes tight bounds on the rate of convergence for this nonlinear dynamics in the so-called
``high-temperature" regime, when the interactions are non-trivial but relatively weak.  Specifically, the condition we require is
that $\max_{x\in V}\sum_{y\in V} |J_{xy}| \le\delta_0$ for some absolute constant~$\delta_0>0$, i.e., the
aggregated strength of all interactions at any given vertex is not too large.  
This condition mirrors the standard Dobrushin condition for Glauber dynamics,
which gives a non-trivial sufficient condition for rapid mixing (see, e.g.,~\cite{Weitz}).
We state this result in the following theorem.
\begin{theorem}\label{thm:main}
In the scenario of Theorem~\ref{thm:convintro}, with the additional assumption that 
$\max_{x\in V}\sum_{y\in V} |J_{xy}| \le\delta_0$ for an absolute constant~$\delta_0>0$,
the rate of convergence of  $p_t$ to $\mu_{{\bf J},{\bf h}}$ is given by 
\begin{equation}
   \Vert p_t - \mu_{{\bf J},{\bf h}}\Vert_{\rm TV} \le Cne^{-ct/n}
\end{equation}
for absolute constants $C,c>0$, where $\Vert\cdot\Vert_{\rm TV}$ denotes total variation distance. 
Thus in particular the time required to achieve variation distance $\varepsilon$ is $t=O(n\log(n/\varepsilon))$.
\end{theorem}
We note that this upper bound on convergence time is (up to constants) the same as for the analogous version
of the genetic recombination model discussed above (where just a single allele is exchanged between the strings at each step)~\cite{Sinetal2},
and is therefore also tight by virtue of the lower bound in the same paper.
That model is equivalent to the trivial case of the Ising model in which there are no interactions ($J_{xy}=0$ for all $x,y$), 
with the $h_x$ determined by the marginal probabilities at each allele~$x$ (and the spins $\pm 1$ identified with the bits $1,0$).
However, as we explain in Section~\ref{subsec:techniques} below, the correlations present in the Ising model
make the analysis much more challenging.

We also consider a ``block" version of the nonlinear dynamics, in which $\sigma,\sigma'$ exchange their spins at a
random subset $\Lambda\subseteq  V$ of vertices (rather than just at a single randomly chosen vertex).  The 
kernel $\kernel{\sigma}{\sigma'}{\tau}{\tau'}$ is again determined by the detailed balance condition~\eqref{eq:detbal},
and the basic convergence result in Theorem~\ref{thm:convintro} still holds. 
Under the same Dobrushin-type condition on the interactions as in Theorem~\ref{thm:main}, we again obtain
a tight bound on the convergence rate:
\begin{theorem}\label{thm:main2}
With the same notation and assumptions as in Theorem~\ref{thm:main}, the variation distance of the
block version of the mass action kinetics for the Ising model satisfies
\begin{equation}
   \Vert p_t - \mu_{{\bf J},{\bf h}}\Vert_{\rm TV} \le Cn^2e^{-ct}
\end{equation}
for absolute constants $C,c>0$.
Thus in particular the time required to achieve variation distance $\varepsilon$ is $t=O(\log(n/\varepsilon))$.
\end{theorem}
Note that convergence here is exponentially faster than in the single-vertex version of Theorem~\ref{thm:main},
reflecting the fact that this version is non-local and changes large portions of the configurations at each step.
Again, the bound of Theorem~\ref{thm:main2} matches the lower bound for zero interaction~\cite{Sinetal2}.

We stress that the goal of this paper is {\it not\/} to design an efficient algorithm for sampling configurations 
of the Ising model.  Such algorithms, based on standard linear Glauber dynamics,
are already known throughout the high-temperature regime.  Rather, our goal is to analyze the rate of
convergence of a natural {\it nonlinear\/} dynamics, for the first time in a model with correlations.  We view
this as a first step towards a better understanding of such dynamics and the techniques needed to understand
them; in addition to their inherent interest, these techniques may lead to algorithmic applications in future.
Our work can be viewed as an extension of the successful application of a TCS lens in the
analysis of mixing times of linear dynamics (Markov chains), which, as is well known, has seen both mathematical and 
algorithmic applications over many years.
However, we should point out that our convergence analysis in Theorems~\ref{thm:main} and~\ref{thm:main2} actually does yield
polynomial time sampling algorithms based on simulation of the respective nonlinear dynamics.  We outline these
algorithms, together with some associated open questions, in Section~\ref{sec:open} at the end of the paper.

We also point out a further interesting algorithmic aspect of our results.  Recall that our nonlinear processes
sample from an Ising Gibbs measure~$\mu_{{\bf J},{\bf h}}$, where the fields~$\bf h$ are determined {\it implicitly\/}
by the marginals at each site.  It is these marginals (not the fields) that are specified by the initial distribution~$p_0$. 
(As far as we are aware, all existing sampling algorithms for the Ising model require the specification of the fields~$\bf h$
rather than the marginals.)  Thus our processes can be viewed as a novel, direct method for constructing a maximum entropy
distribution subject to these marginal constraints, an important problem in its own right (see, e.g., \cite{SV13}).   
Additionally, our processes can be used to learn the fields~$\bf h$ corresponding to given marginals, an inverse problem
that is also of independent interest (see \cite{NZB17} for a survey of such inverse problems): given samples from~$\mu_{{\bf J},{\bf h}}$
produced by the nonlinear dynamics, standard methods can be used to infer the field vector~$\bf h$.

\subsection{Techniques}\label{subsec:techniques}
We begin by describing the earlier approach of~\cite{Sinetal2} to analyzing the rate of convergence of the simpler
population genetics dynamics, which corresponds to the trivial case of the Ising model with no interactions (${\bf J}=0$).  Since
the equilibrium distribution here consists of independent bits, the analysis is relatively straightforward once one
observes the following insight.  The derivation of
an individual~$\sigma$ at time~$t$ can be viewed as a binary tree going backwards in time, in which each individual inherits
a random subset of its bits from each of its two parents according to the random crossover subset~$\Lambda$.  
We may therefore follow the derivation of the $n$ bits in~$\sigma$ back in time, until each of these bits is derived
from a {\it distinct\/} individual at time~0.  At that point we can deduce that the bits of~$\sigma$ are independently
sampled from their respective marginal distributions, so $\sigma$ is in equilibrium.  The analysis therefore reduces to the
question of how many steps are needed until all the~$n$ bits are separated, or ``fragmented", under the repeated 
action of partition by the random crossover subset~$\Lambda$, which in turn is a straightforward combinatorial calculation.
(See Section~\ref{sec:fragment} for a more detailed description of this process.)

In the case where correlations are present in the equilibrium distribution, as in the Ising model, the above analysis
breaks down because it is no longer sufficient to consider only fragmentation of the bits: indeed, the process must
involve not only the breakdown of correlations in the initial distribution, but also, crucially, the {\it creation\/} of the correct
equilibrium correlations as mandated by the Gibbs distribution~\eqref{eq:Gibbs}.  Moreover, the process by which an
individual inherits bits from its parents is no longer independent of the parents, but dictated by a complex function of
both parents.  

To account for this, we appeal to
the {\it information percolation\/} framework developed by Lubetzky and Sly~\cite{lubetzky2017universality} in the
context of Glauber dynamics for the Ising model.  This framework suggests that we keep track of a ``dependence cluster"
going back in time, which records the neighboring spins that have influenced each spin in our current configuration.  
In the linear setting of~\cite{lubetzky2017universality}, it can be shown (under a similar high-temperature
assumption to ours) that this cluster is dominated by a subcritical branching process and thus will die out with large probability:
the equilibrium correlations are then implicitly encoded by the history of this process.  The time until the process dies out 
gives a bound on the mixing time.

In our nonlinear setting, the dependence clusters are no longer describable in terms of a simple branching process, but
rather by a new type of process that combines branching with fragmentation, a process we refer to as ``fragmentation plus
noise."  The first main ingredient of our analysis is the precise construction of such a process and the proof that it encodes the
complex dependence structure of the nonlinear dynamics. The second main ingredient is the proof that, under the 
high-temperature assumption $\max_{x\in V}\sum_{y\in V} |J_{xy}| \le\delta_0$, the fragmentation plus noise process
has a subcritical behavior and therefore dies out with large probability on a suitable time scale.   To establish this
latter fact, we introduce a non-standard form of  ``high-temperature expansion" for the dependence structure obtained
by a coupling with non-uniform Erd\H{o}s-R\'enyi random graphs.
We refer the reader to Section~\ref{subsec:sketch} for a more technical high-level description of these ideas.

\subsection{Related work}\label{subsec:related}
The framework of mass action kinetics was first introduced in the chemical reaction networks literature, most
notably in the landmark paper of Horn and Jackson~\cite{HornJackson}.  The dynamics defined in~\eqref{eq:massaction}
above, generalized to allow interactions between arbitrary numbers of types, models chemical processes in an
obvious way.  The recent monograph of Feinberg~\cite{FeinbergBook} describes the state of the art in the area.
While almost nothing is known about rates of convergence, much
effort has been devoted to a proof of the Global Attractor Conjecture mentioned earlier; indeed, the original
paper~\cite{HornJackson} contained this as a theorem, but the authors later retracted it and restated it as a
conjecture~\cite{Horn74}.
Since then there have been numerous attempts at a proof, including recent papers that
handle various special cases, typically based on rather severe structural conditions on the set of reactions
(such as forming a single connected component)---see, e.g., \cite{CNP,Pantea, GMS, AGHMR, ACKN, DZ}.  
The Conjecture remains open.

Several other important classes of dynamics fit into the mass action framework, the most classical of which
is Boltzmann's model of an ideal gas~\cite{Boltz}, where the types are momentum values of the molecules and interactions
correspond to (randomized) collisions between pairs of them\footnote{Here the set of types is continuous
rather than discrete, as in our setting.}.  In this case the dynamics~\eqref{eq:massaction} is the so-called
{\it Boltzmann equation}, which remains a major object of study in mathematical physics today (see~\cite{Villanisurvey}
for a survey).
Another famous example is Hardy-Weinberg recombination in population genetics, as described earlier.
The same dynamics can also be used to model the ``recombination" step in genetic algorithms~\cite{goldberg,Mitchell},
where two (or more) candidate solutions to a combinatorial optimization problem are combined to produce
new solutions; here the mass action dynamics is typically combined with a ``selection" operator that weeds
out less desirable solutions.  The viewpoint taken in this paper, where mass action kinetics are viewed as
a nonlinear version of the Markov chain Monte Carlo method for sampling combinatorial structures
from a given distribution, was first proposed in~\cite{Sinetal} and later explored in a different context in~\cite{SV13}.
Negative results on the worst-case computational complexity of simulating mass action kinetics were derived in~\cite{ARV}.
Mass action kinetics resembles certain more refined MCMC approaches that are used in practice for numerical simulations
of spin systems such as the Ising model, including replica Monte Carlo~\cite{SW86},
parallel tempering~\cite{tempering} and---most closely---cluster Monte Carlo~\cite{Houdayer}; these are in 
fact linear dynamics, but because they update a set of configurations in a dependent way, they can be
viewed as approximate finite realizations of nonlinear dynamics.

As discussed earlier, there are almost no quantitative results on the rate of convergence of mass action kinetics
in combinatorial settings, the main exception being genetic recombination~\cite{Sinetal2,CapSin,CapPar}.  There is a vast
and still evolving
literature on convergence rates of the Boltzmann equation (see, e.g.,~\cite{kac1956foundations,Villanisurvey,mischler2013kac,DGR}), 
which however is tailored to the specifics of that model and does not seem to generalize.  

Finally we mention the more general class of so-called ``nonlinear Markov chains", which (in discrete time) are stochastic
processes~$(X_t)$ in which the distribution of~$X_t$ depends not only on the previous state~$X_{t-1}$ but also (in an arbitrarily
complex way) on the  {\it distribution\/} of~$X_{t-1}$.  (Mass action kinetics as discussed above
is a particular example where the dependence is quadratic.)  This class was formally introduced by McKean~\cite{mckean1966class},
who studied the continuous time version and its deep relationships to nonlinear parabolic equations, including the Boltzmann
equation, and other kinetic models.  These nonlinear systems arise naturally as the limit of large finite mean-field particle
systems through the so-called ``propagation of chaos"   \cite{kac1956foundations,kolokoltsov2010nonlinear,sznitman1991topics,villani2002review,carrillo2003kinetic,Villanietal,cattiaux2008probabilistic,mischler2013kac,budhiraja2015local,ChaintronDiez22-1,ChaintronDiez22-2}.
Several works have been devoted to the development of nonlinear Markov chain Monte Carlo methods \cite{DelMoraletal,muzychka2011class,erbar2020entropic,butkovsky2014ergodic,shchegolev2022convergence,vuckovicnonlinear}, 
but in contrast with the classical (linear) Markov chain framework, quantitative results on convergence to stationarity 
are very limited and difficult to obtain.

\subsection{Organization of the paper}\label{subsec:org}
In Section~\ref{sec:prelim} we formally define both of our nonlinear dynamics and establish some of their
basic properties, including a proof of a general convergence result, Theorem~\ref{th:genconv}, which we then use to prove
Theorem~\ref{thm:convintro}. In Section~\ref{sec:nonlinearbd}
we proceed with our quantitative analysis of the convergence rate for the nonlinear block dynamics, culminating in a 
proof of Theorem~\ref{thm:main2}; prior to embarking on the details, we provide in Section~\ref{subsec:sketch}
a more technical, high-level sketch of our approach.  In Section~\ref{sec:Glauber} we apply a similar approach, though 
substantially different in detail, to analyze the nonlinear Glauber dynamics and prove Theorem~\ref{thm:main}.
We conclude with some additional observations, extensions and open problems in Section~\ref{sec:open}.


\section{Preliminaries}\label{sec:prelim}

\subsection{The Ising model}
We recall from the introduction the definition~\eqref{eq:Gibbs} of the Ising model via its Gibbs distribution~$\mu_{\bf J,h}$.  Note that we
allow arbitrary edge-dependent interactions ${\bf J} = \{J_{xy}\}_{x,y\in V}$ and arbitrary external
fields ${\bf h} = \{h_x\}_{x\in V}$.  When all the external fields are zero we write simply $\mu_{\bf J}$.
We identify the set of vertices (or sites)~$V$ with~$[n]$, and denote
the set of spin configurations by $\O=\{\pm 1\}^{[n]}$.

\begin{remark}\label{rem:bc}
For simplicity we have taken a model with no boundary conditions. However, there is no difficulty in extending our analysis and
results to the case of a Gibbs measure with arbitrary boundary conditions, i.e., when the spins in some subset $V_0\subseteq V$ are
pinned to given values~$\pm 1$. This generalization can be easily achieved by taking  limits $h_x\to\pm\infty$ for all $x\in V_0$ that
are pinned to the values~$\pm 1$, respectively.
\end{remark}

\subsection{The nonlinear dynamics}
Let $\cP(\O)$ denote the set of probability measures on $\O$, and $\cP_+(\O)\subseteq\cP(\O)$ 
the set of measures with full support, i.e.\ $p\in\cP_+(\O)$ iff $p(\o)>0$ for all $\o\in\O$.
We consider nonlinear (mass action) dynamics $\{p_t\}_{t\ge 0}$ on the set of types~$\O$ as defined in~\eqref{eq:massaction} of the
introduction, with some kernel~$\cQ$
satisfying the symmetry $\kernel{\sigma}{\sigma'}{\tau}{\tau'} = \kernel{\sigma'}{\sigma}{\tau'}{\tau}$.
We shall also assume that $\cQ$ satisfies the reversibility condition
\begin{equation}
\label{cw2qq}
 \mu(\sigma) \mu(\sigma') \kernel{\sigma}{\sigma'}{\tau}{\tau'} = \mu(\tau) \mu(\tau') \kernel{\tau}{\tau'}{\sigma}{\sigma'}\,, \qquad \forall\, \sigma,\sigma',\tau,\tau',
\end{equation}
for some $\mu\in\cP_+(\O)$. 

We view this dynamics as a dynamical system $p\mapsto T_t(p)$, where $T_0(p)=p\in\cP(\O)$ is the initial distribution,
$T_t(p)\in\cP(\O)$ is the distribution after $t$ steps, and $T_t(p) = T_{t-1}(p) \circ T_{t-1}(p)$.
Here one step of the dynamics is defined, as in~\eqref{eq:massaction}, by
\begin{equation}\label{cwnew}
p \mapsto p\circ p := \sum_{\si,\si'\!,\tau'}p(\si)p(\si')\kernel{\sigma}{\sigma'}{\tau}{\tau'}.
\end{equation}
It will be convenient later to write~\eqref{cwnew} in the equivalent form
\begin{equation}
\label{cw2}
   p \mapsto p\circ p := \sum_{\si,\si'}p(\si)p(\si')Q(\cdot\tc\si,\si'),
\end{equation}
where, for fixed $\si,\si'\in\O$,  the distribution $Q(\cdot\tc\si,\si')\in\cP(\O)$ is defined by
\begin{equation}
\label{cw2q}
Q(\t\tc\si,\si') := \sum_{\t'\in\O}\kernel{\sigma}{\sigma'}{\tau}{\tau'}.
\end{equation}
In this paper we take $\mu=\mu_{\bf J,h}$ as the Ising measure~\eqref{eq:Gibbs} and consider two natural choices of the kernel $\cQ$
that satisfy~\eqref{cw2qq}, which we now describe.

\subsubsection{Nonlinear block dynamics} The first model, which we refer to as the {\em nonlinear block dynamics}, corresponds
to interactions in which a pair of configurations $(\sigma,\sigma')$ exchange their spins at an arbitrary, randomly chosen subset
$\L\subseteq  [n]$ of sites, i.e.,
\begin{align}\label{eq:eqclass}
(\si,\si') \mapsto (\si'_\L\si_{\L^c},\si_\L\si'_{\L^c}),
\end{align}
where $\si_\L\si'_{\L^c}$ denotes the element of $\Omega$ with entries $\si_x$ for $x \in \L$ and $\si'_x$ for $x \in \L^c=[n]\setminus \L$. 
Here, to ensure reversibility, the set $\L$ is chosen with probability proportional to $ \mu(\si'_\L\si_{\L^c})\mu(\si_\L\si'_{\L^c})$. 
Thus the associated kernel is defined as 
\begin{equation}
\label{pq20}
\kernelJ{\si}{\si'}{\t}{\t'} =
\frac{\sum_{\L\subseteq V}\mu(\si'_\L\si_{\L^c})\mu(\si_\L\si'_{\L^c})\ind_{\t=\si'_\L\si_{\L^c}}\ind_{\t'=\si_\L\si'_{\L^c}}}
{\sum_{A\subseteq V}\mu(\si'_A\si_{A^c})\mu(\si_A\si'_{A^c})}.
\end{equation}
Note that transitions of the form \eqref{eq:eqclass} can only produce pairs $(\t,\t')$ that belong to the equivalence class
\begin{align}\label{eq:eqclass2}
\cC(\si,\si')=\left\{(\si'_\L\si_{\L^c},\si_\L\si'_{\L^c}),\;\L\subseteq V\right\}.
\end{align}
Thus the kernel~\eqref{pq20} defines a (linear) Markov chain on the pair space $\O\times\O$ that is in general not irreducible,
and whose communicating classes are precisely~$\cC(\si,\si')$.
Note also that the kernel $\cQ_{\bf J}$ depends on $\mu=\mu_{\bf J,h}$ only through the interaction ${\bf J}$ and is insensitive
to the choice of fields $\bf h$. Indeed, once $\si,\si'$ are given, then for any $(\t,\t')\in \cC(\si,\si')$ and $(\eta,\eta')\in \cC(\si,\si')$ one has 
\begin{align}\label{eq:eqclass3}
\frac{\mu_{\bf J,h}(\t)\mu_{\bf J,h}(\t')}{\mu_{\bf J,h}(\eta)\mu_{\bf J,h}(\eta')}=
\frac{\mu_{\bf J,h'}(\t)\mu_{\bf J,h'}(\t')}{\mu_{\bf J,h'}(\eta)\mu_{\bf J,h'}(\eta')}\,, \qquad \forall\;{\bf
h,h'}\in\bbR^n.
\end{align}
Thus, w.l.o.g., we may take ${\bf h}=0$, and $\mu=\mu_{\bf J}=\mu_{\bf J,0}$, in the definition of the kernel~\eqref{pq20}. The kernel $\cQ_{\bf J}$ in \eqref{pq20} is
an example of a  so-called ``folding" transformation~\cite{BG}.

Observe that, for all ${\bf h}\in\bbR^n$, the reversibility  condition~\eqref{cw2qq} holds in the form
\begin{equation}
\label{cw2qqla}
\mu_{\bf J,h}(\si)\mu_{\bf J,h}(\si')\kernelJ{\si}{\si'}{\t}{\t'} = \mu_{\bf J,h}(\t)\mu_{\bf J,h}(\t')\kernelJ{\t}{\t'}{\si}{\si'},
\end{equation}
for all $\si,\si',\t,\t'\in\O$. Thus, for a fixed interaction ${\bf J}$, the kernel $\cQ_{\bf J}$ is reversible w.r.t.\ all measures $\{\mu_{\bf J,h}$, ${\bf h}\in\bbR^n\}$.
In particular, all these measures are stationary for the dynamics~\eqref{cwnew}, i.e., 
\begin{align}
\label{eq:inv}
\mu_{\bf J,h}\circ \mu_{\bf J,h}= \mu_{\bf J,h}\, , \qquad \forall\; {\bf h}\in\bbR^n,
\end{align}
as can be easily checked from reversibility.

We note that in the case of zero interactions, i.e.,  ${\bf J}=0$, the nonlinear block dynamics reduces to the uniform crossover model
from population genetics \cite{Sinetal2,CapPar}. In this case, the stationary distributions $\mu_{\bf J,h}$ are just product measures over
spins with marginals determined by~$\bf h$.

\subsubsection{Nonlinear Glauber dynamics} In our second model, the configurations $\sigma,\sigma'$ exchange their spins at a
{\it single\/} randomly chosen site~$x\in [n]$, i.e., 
\begin{align}\label{eq:eqclassx}
(\si,\si') \mapsto (\si'_x\si_{[n]\setminus \{x\}},\si_x\si'_{[n]\setminus \{x\}}).
\end{align}
By analogy with the familiar Glauber dynamics (a Markov chain that updates the spin at one site
in each step), we refer to this as the {\em nonlinear Glauber dynamics}.  As usual, to ensure reversibility
w.r.t.~$\mu_{\bf J,h}$, we need to perform such an exchange with an appropriate probability $\a_x(\si,\si')$.  
Specifically, we use the generic dynamics~\eqref{cwnew} with the kernel
\begin{equation}
\label{pq22}
\kernelJ{\si}{\si'}{\t}{\t'} =
\frac1{n}\sum_{x\in V} \kernelJx{\si}{\si'}{\t}{\t'},
\end{equation}
where
\begin{equation*}
\kernelJx{\si}{\si'}{\t}{\t'} :=\a_x(\si,\si')\ind_{\t=\si'_x\si_{[n]\setminus \{x\}}}\ind_{\t'=\si_x\si'_{[n]\setminus \{x\}}}+ (1-\a_x(\si,\si'))
\ind_{\t=\si}\ind_{\t'=\si'},
\end{equation*}
and
\begin{equation*}
\a_x(\si,\si') := \frac{\mu(\si'_x\si_{[n]\setminus \{x\}})\mu(\si_x\si'_{[n]\setminus \{x\}})}{\mu(\si'_x\si_{[n]\setminus \{x\}})\mu(\si_x\si'_{[n]\setminus \{x\}})+\mu(\si)\mu(\si')}.
\end{equation*}
Once again, the Markov chain on pairs $\O\times\O$ defined by the kernel \eqref{pq22} is not irreducible, and the kernel $\cQ_{\bf J}$
depends on $\mu=\mu_{\bf J,h}$ only through the interaction ${\bf J}$. As in \eqref{cw2qqla}-\eqref{eq:inv}, reversibility and stationarity
of all measures $\mu_{\bf J,h}$ can be easily checked.
 
\subsection{Conservation laws}
In both dynamics defined above, the map $p\mapsto p\circ p$ conserves the marginal probabilities of spins at every vertex,
i.e., for every $x\in[n]$, and for any $p\in\cP(\O)$, one has 
\begin{equation}
\label{consepqu}
(p\circ p)_x=p_x,
\end{equation} where $p_x(a) :=p (\si_x=a)$, $a\in\{-1,1\}$, denotes the marginal of $p$ at $x$.
It is convenient to state the following stronger property.  Let us define the commutative 
convolution product of two distributions $p,q\in\cP(\O)$ by
\begin{equation}
\label{consepq}
 p\circ q :=  \frac12\sum_{\si,\si'} \bigl(p(\si)q(\si')+p(\si')q(\si)\bigr)Q(\cdot\tc\si,\si'),
\end{equation}
where $Q$ is defined by \eqref{cw2q} and \eqref{pq20} for the nonlinear block dynamics and by \eqref{cw2q} and \eqref{pq22} for the nonlinear Glauber dynamics, respectively. Note that the notation \eqref{consepq} is consistent with  \eqref{cw2}.
\begin{lemma}
\label{lem:conser}
Both of the above dynamics satisfy
\begin{equation}
\label{consepqa}
 (p\circ q)_x=\tfrac12\,(p_x +q_x).
\end{equation}
In particular, the conservation law \eqref{consepqu} holds.
\end{lemma}
\begin{proof}
The proof is a consequence of the fact that, in both dynamics, spins are 
simply exchanged between $\sigma,\sigma'$, as well as the symmetry of $\cQ_{\bf J}$. 
For the kernel \eqref{pq20} we compute
\begin{align}
\label{pq20ax}
\sum_{\t,\t'}\kernelJ{\si}{\si'}{\t}{\t'}\ind_{\t_x=a}&=
\frac{\sum_{\L\subseteq V}\mu(\si'_\L\si_{\L^c})\mu(\si_\L\si'_{\L^c})(\ind_{\si'_x=a,\L\ni x}+\ind_{\si_x=a,\L\not\ni x})}
{\sum_{A\subseteq V}\mu(\si'_A\si_{A^c})\mu(\si_A\si'_{A^c})}
\\ & = \ind_{\si_x=a}  + \frac{\sum_{\L\subseteq V}\mu(\si'_\L\si_{\L^c})\mu(\si_\L\si'_{\L^c})(\ind_{\si'_x=a}-\ind_{\si_x=a})\ind_{\L\ni x}}
{\sum_{A\subseteq V}\mu(\si'_A\si_{A^c})\mu(\si_A\si'_{A^c})}.
\end{align}
By symmetry,
\begin{align}
\label{pq20aax}
\frac12\sum_{\t,\t'}\bigl(\kernelJ{\si}{\si'}{\t}{\t'}+\kernelJ{\si'}{\si}{\t}{\t'}\bigr)\ind_{\t_x=a} = 
\frac12(\ind_{\si_x=a} + \ind_{\si'_x=a}).
\end{align}
In conclusion,
\begin{align}
\label{consepqas}
 (p\circ q)_x(a)& =  \frac12\sum_{\si,\si'} \bigl(p(\si)q(\si')+p(\si')q(\si)\bigr)Q(\cdot\tc\si,\si')\\& = 
 \frac12\sum_{\si,\si'} p(\si)q(\si')\sum_{\t,\t'}\bigl(\kernelJ{\si}{\si'}{\t}{\t'}+\kernelJ{\si'}{\si}{\t}{\t'}\bigr)\ind_{\t_x=a} \\& = 
 \frac12\sum_{\si,\si'} p(\si)q(\si')(\ind_{\si_x=a} + \ind_{\si'_x=a}) =  \frac12(p_x(a)+q_x(a)). 
 \end{align}
 This proves the lemma for the nonlinear block dynamics. To prove it for  the nonlinear Glauber dynamics, observe that in this case by \eqref{pq22} one has  
\begin{align}
\label{pq20axxo}
\sum_{\t,\t'}\kernelJ{\si}{\si'}{\t}{\t'}\ind_{\t_x=a}
= \ind_{\si_x=a} + \frac{1}n\,\a_x(\si,\si')(\ind_{\si'_x=a}-\ind_{\si_x=a}).
\end{align}
 By the symmetry $\a_x(\si,\si')=\a_x(\si',\si)$,
\begin{align}
\label{pq20aoax}
\frac12\sum_{\t,\t'}\bigl(\kernelJ{\si}{\si'}{\t}{\t'}+\kernelJ{\si'}{\si}{\t}{\t'}\bigr)\ind_{\t_x=a} = 
\frac12(\ind_{\si_x=a} + \ind_{\si'_x=a}),
\end{align}
and the conclusion follows as in \eqref{consepqas}. 
\end{proof}

\subsection{The derivation tree and fragmentation}\label{sec:fragment}
Throughout the paper, the following view of the nonlinear dynamics will be central.  By definition, 
$T_t(p)$ is the result of repeated pairwise interactions and can be represented as the distribution at the root of  a binary ``derivation" tree, where each leaf is equipped with the distribution $p\in\cP(\O)$, and recursively, starting from the leaves, each internal node is assigned the distribution $p_1\circ p_2$ where $p_1,p_2$ represent the distributions assigned  to the left and right descendants of that node; see Figure~\ref{fig1} for a schematic picture of the case $t=2$.

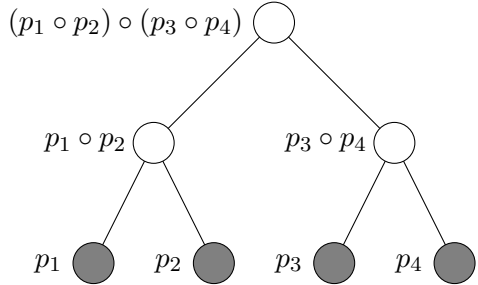
\begin{figure}[h]
\center

\begin{tikzpicture}[scale=0.8]
    
    \draw  (1,-1) -- (2,1);
    \draw  (3,-1) -- (2,1);
   \draw  (2,1) -- (4,3);
    \draw  (5,-1) -- (6,1);
    \draw  (7,-1) -- (6,1);
    \draw  (6,1) -- (4,3);
    
    \node[shape=circle, draw=black, fill = white, scale = 1.5]  at (4,3) {}; 
    \node[shape=circle, draw=black, fill = white, scale = 1.5]  at (2,1) {};
    \node[shape=circle, draw=black, fill = gray, scale = 1.5]  at (1,-1) {};
    \node[shape=circle, draw=black, fill = gray, scale = 1.5]  at (3,-1) {};
    \node[shape=circle, draw=black, fill = white, scale = 1.5]  at (6,1) {};
    \node[shape=circle, draw=black, fill = gray, scale = 1.5]  at (5,-1) {};
    \node[shape=circle, draw=black, fill = gray, scale = 1.5]  at (7,-1) {};
    
    \node at (.25,-1) {$p_1$};
    \node at (2.25,-1) {$p_2$};
    \node at (4.25,-1) {$p_3$};
    \node at (6.25,-1) {$p_4$};
    \node at (1,1) {$p_1 \circ p_2\;\;$};
    \node at (5,1) {$p_3 \circ p_4\;\;$};
    
    \node at (2,3) {$(p_1\circ p_2) \circ (p_3 \circ p_4)\;\;\;\;\;\;\;$};

\end{tikzpicture}
\caption{Graphical representation of the distribution $(p_1\circ p_2) \circ (p_3 \circ p_4)$ at time $t=2$ when each leaf $i=1,\dots,4$ is equipped with distribution~$p_i$. When $p_i = p\;\forall i$, the distribution at the root is $T_2(p)$. }
\label{fig1}
\end{figure}

We focus now on the simple case ${\bf J}=0$, i.e., no correlations between spins.
Under block dynamics, the configuration $\tau$ at the root of the tree (at time~$t$) is constructed according to the random 
partition $(\L,\L^c)$ of~$V$,
which is equivalent to drawing each spin $\tau_x$ from the configuration at the left
or right child node with probability~${\frac 1 2}$, independently for each site $x\in V$.  
Continuing down the tree in the same fashion, we see that each spin at the root is drawn
from one of the $2^t$ leaves (at time~0), independently and uniformly at random.  Thus this
process induces a {\it partition\/} of the sites~$V$ into $2^t$ disjoint subsets (some of which
may be empty), where the $\ell$th subset consists of those sites that draw their spin from
the configuration at leaf~$\ell$.

Now let $\cA$ denote the event that none of the $2^t$ subsets in this partition contains
more than one site; equivalently, each spin in~$\tau$ is drawn from a {\it distinct\/}
leaf.  We call $\cA$ the ``complete fragmentation" event.  Note that, conditional on~$\cA$,
the distribution of the configuration~$\tau$ at the root is just the product $\pi:=\otimes_{x\in V}p_x$,
since there are no remaining correlations between spins.  Hence we may write 
\begin{align}\label{eq51}
 T_t(p)= \nu(\cA)\, \pi + \nu(\cA^c) \,q(t),
\end{align}
for some other distribution~$q(t)$, where $\nu$ denotes the uniform distribution over 
all $2^t-1$ independent random subsets~$\L$ occurring in the tree.  

We can use~\eqref{eq51} to obtain an upper bound on the convergence time for the nonlinear block dynamics when ${\bf J}=0$,
as was done in~\cite{Sinetal2}.  First, we claim that $\nu(\cA^c)\leq {n\choose 2}2^{-t}$.  To see this, 
note that for any given pair of distinct sites $x,y\in V$,
the probability that $x,y$ are not separated after $t$ levels of the successive partitioning process 
is~$2^{-t}$, and then take a union bound over pairs.  Hence by~\eqref{eq51}, taking $t=O(\log(n/\e))$
ensures that $\Vert \widetilde T_t(p) - \pi\Vert_{\rm TV} \le\e$, so the convergence time is $O(\log(n/\e))$.

For the nonlinear Glauber dynamics with ${\bf J}=0$ a similar analysis applies, except that now each node in the tree
chooses one random spin from the left child and the remainder from the right child.  The complete fragmentation 
event~$\cA$ now corresponds to isolating each of the $n$ bits in this process, which is just a
coupon-collecting event for $n=|V|$ coupons.  Thus we have $\nu(\cA^c)\leq n(1-\frac{1}{n})^t\le ne^{-t/n}$
(where now $\nu$ denotes the uniform distribution over all $2^t-1$ independent choices of random spins 
occurring in the tree),
which by~\eqref{eq51} implies a convergence time of $O(n\log(n/\e))$.

When ${\bf J}\ne 0$, so that correlations are present, it is no longer possible to reduce the analysis of 
convergence to the above simple fragmentation process, because the mechanism by which a configuration
inherits spins from its parents depends on the actual configurations at the parent nodes.
Thus to prove Theorems~\ref{thm:main} and~\ref{thm:main2} we will need to augment the simple derivation process above
to obtain a more complex process that we call ``fragmentation with noise"
(see Sections~\ref{sec:nonlinearbd} and~\ref{sec:Glauber}).

\subsection{Irreducibility}\label{sec:irred}
Next we establish a rough lower bound on the probability of any configuration after a sufficiently long time.
This observation will be key to our proof of convergence in the next subsection.  We note that the lack of such a
lower bound is the principle obstacle to proving the Global Attractor Conjecture for general reversible mass action systems.  

The following definition captures the desired property.
\begin{definition}\label{def:irred}
We say that an initial distribution $p\in\cP(\O)$ is {\em irreducible} for a given mass action system with kernel~$\cQ$ if there
exist $\e>0$ and~$t_0$ such that $T_t(p)(\tau)\ge\e$ for all 
$t\ge t_0$ and all $\tau\in\Omega$.
\end{definition}
\noindent
Thus irreducibility says that the trajectory of the dynamics
starting from~$p$ eventually remains bounded away from the
boundary of the simplex.  Note that, under this definition, irreduciblity is a property of initial distributions:
i.e., for a given kernel~$\cQ$, some initial distributions may be irreducible while others are not.

In what follows we shall assume that the initial distribution~$p$ of our dynamical system has {\it nondegenerate
marginals}, by which we mean that there exists $\d>0$ such that
\begin{align}\label{margo}
\min_{x\in V}\min_{a\in\{-1,+1\}}p_x(a)\geq \d>0.
\end{align}
This is actually no loss of generality since one can otherwise restrict to the nondegenerate spins and consider the degenerate spins 
as a fixed boundary condition, or pinning; see Remarks~\ref{rem:bc} and~\ref{rem:bc2}.

\begin{lemma}\label{lem:irred} 
For any interaction matrix ${\bf J}$, any initial distribution $p\in\cP(\O)$ with nondegenerate marginals is irreducible
for both the nonlinear block dynamics and the nonlinear Glauber dynamics.
\end{lemma}


\begin{proof}
Let us first consider the nonlinear block dynamics \eqref{pq20}. We note that for a fixed interaction matrix ${\bf J}\in\bbR^{n\times n}$,
there exists a constant $\d_{\bf J}>0$ such that
\begin{align}
\mu_{\bf J}(\si)\mu_{\bf J}(\si')\geq \d_{\bf J}\,,\qquad \forall\, \si,\si'\in\O.
\end{align}
(Recall that $\cQ_{\bf J,h}=\cQ_{\bf J}$ does not  depend on~$\bf h$, so since we are discussing a property of~$\cQ_{\bf J}$ we may take ${\bf h}=0$.)
It follows from \eqref{pq20} that
\begin{align}
\label{pq20dj}
\kernelJ{\si}{\si'}{\t}{\t'}=\d_{\bf J}\,2^{-n}\sum_{\L\subseteq V}
\ind_{\t=\si'_\L\si_{\L^c}}\ind_{\t'=\si_\L\si'_{\L^c}} + (1-\d_{\bf J}) \kernelJhat{\si}{\si'}{\t}{\t'},
\end{align}
where,
for each $\si,\si'$, $\kernelJhat{\si}{\si'}{\cdot}{\cdot}\in\cP(\O\times\O)$ is some probability distribution that we do not need to describe explicitly. 
Therefore, from \eqref{cw2},
\begin{align}
\label{cw2irre}
(p\circ p) = \d_{\bf J} 2^{-n}\sum_{\L\subseteq V} (p_\L\otimes p_{\L^c}) + (1-\d_{\bf J})\widehat\Phi(p),
\end{align}
where $(p_\L\otimes p_{\L^c})(\t)=p_\L(\t_\L)p_{\L^c}(\t_{\L^c})$ denotes the product of marginals of $p$ on $\L$, $\L^c$,
and $\widehat\Phi(p)\in\cP(\O)$ is some new distribution. 
We may interpret~\eqref{cw2irre}
as saying that the outcome of each interaction is, with probability (at least)~$\d_{\bf J}$, equal to the factorized distribution $p_\L\otimes p_{\L^c}$,
where $\L$ is chosen u.a.r.\ among all subsets of~$V$. 
Now note that the map $p\mapsto 2^{-n}\sum_{\L\subseteq V} p_\L\otimes p_{\L^c}$ corresponds to one step
of the block dynamics when ${\bf J}=0$.
Writing $\widetilde T_t(p)$ for the $t$-step evolution of this ${\bf J}=0$ dynamics,
and noting from the tree representation that the construction of $T_t(p)$ involves $N:=2^{t}-1$ interactions, 
we may rewrite~\eqref{cw2irre} as
\begin{align}
\label{cw2irres}
 T_t(p)= \d_{\bf J}^N \widetilde T_t(p)  + (1-\d_{\bf J}^N)\widehat T_t(p),
\end{align}
where $\widehat T_t(p)\in\cP(\O)$ is again some other distribution that we will not describe.

To complete the argument, we appeal to the analysis of the ${\bf J}=0$ case from the previous subsection.
Specifically, we use equation~\eqref{eq51} together with the analysis of the complete fragmentation event~$\cA$,
which implies that $\nu(\cA)\ge\frac 1 2$ for $t\ge t_0=\lceil 2\log_2 n\rceil$,
to deduce that $\widetilde T_{t_0}(p)(\tau)\ge {\frac12}\pi(\tau)$ for all $\tau\in\O$.
Here $\pi$ is the stationary distribution for the ${\bf J}=0$ case, which is just the product $\pi:=\otimes_{x\in V}p_x$,
and the assumption~\eqref{margo} implies $\pi(\tau)\ge\delta^n$ for all~$\tau$.  Plugging these observations
into~\eqref{cw2irres} gives
\begin{align}\label{eq52}
 T_{t_0}(p)(\t)\geq  \tfrac12\,\d_{\bf J}^{N_0}\d^n\,,\qquad \forall\, \t\in\O,
\end{align}
where $N_0=2^{t_0}-1$.
This proves the claim for $t=t_0$ with $\e=\tfrac12\,\d_{\bf J}^{N_0}\d^n$. To prove it for all  $t\geq t_0$, observe that if $t\geq t_0$ then $ T_{t}(p)= T_{t_0}( T_{t-t_0}(p))$. 
Since by Lemma~\ref{lem:conser}, $T_{t-t_0}(p)$ has the same marginals as~$p$, we can apply the bound~\eqref{margo} to $ T_{t-t_0}(p)$ with the same constant $\d$. Using \eqref{eq52} with $p$ replaced by  $T_{t-t_0}(p)$ yields $T_{t}(p)(\t)\geq \e$ with the same $\e$ for all $t\geq t_0$. 
This completes the proof of the lemma for the nonlinear block dynamics.

To prove it for the nonlinear Glauber dynamics we follow the same reasoning.
We first observe that for some $\d_{\bf J}>0$ one has \[
\a_x(\si,\si')\geq \d_{\bf J}\,,\qquad \forall\, x\in V\,,\;\si,\si'\in\O
\]
and therefore~\eqref{cw2irre} now takes the form
\begin{align}
p\circ p&= \d_{\bf J}\frac1n\sum_{x\in V}p_x\otimes p_{[n]\setminus\{x\}} + (1-\d_{\bf J})\widehat\Phi(p),
\end{align}
for some new measure $\widehat\Phi(p)\in\cP(\O)$.
We may again write the expression~\eqref{cw2irres}, where now $\widetilde T_t(p)$ is the $t$-step evolution of the
${\bf J}=0$ nonlinear Glauber dynamics,
$p\mapsto p_x\otimes p_{[n]\setminus\{x\}}$ with $x\in V$ chosen uniformly at random.
(Incidentally, it is interesting to note that, in contrast with the block dynamics, the process $\widetilde T_t(p)$ here is actually linear in~$p$, 
since the marginals $p_x$ are constants of the motion.)  

Following the same reasoning as above, again using equation~\eqref{eq51} from the previous subsection
together with the analysis of the complete fragmentation event~$\cA$ (which in this case corresponds to coupon-collecting)
leads to $\nu(\cA)\ge{\frac12}$ for $t\ge t_0=n\lceil 1+\log n\rceil$, which in turn with the assumption~\eqref{margo} implies
$\widetilde T_{t_0}(p)(\t)\geq \d^n/2$ for all $\tau\in\Omega$.  It follows that
\begin{align}\label{eq52a}
 T_{t_0}(p)(\t)\geq  \tfrac12\,\d_{\bf J}^{t_0}\d^n\,,\qquad\forall\, \t\in\O.
\end{align}
The desired conclusion for all $t\geq t_0$ follows as in the block dynamics case.  
\end{proof}

\begin{remark}\label{rem:irr}
For (linear) Markov chains, an irreducibility statement as in Lemma~\ref{lem:irred} immediately implies exponential 
convergence to stationarity (at some possibly very slow rate).
For the nonlinear dynamics considered here, however, it is an open question whether this holds.  We will see in the next 
subsection that Lemma~\ref{lem:irred} is enough to ensure convergence, though without any information about the rate; one
reason why this may be delicate is the fact that, for our nonlinear dynamics, the total variation distance to stationarity is 
in general not monotonically decreasing~\cite[Remark 2.7]{CLL}.
However, our main results (Theorems~\ref{thm:main}  and~\ref{thm:main2} in the introduction) do imply exponential convergence in the
high-temperature regime (i.e., when the interactions in~$\bf J$ are sufficiently weak).
\end{remark}

\subsection{Convergence to stationarity}\label{sec:cgce}
As we have seen, reversibility implies that, for a fixed interaction matrix~${\bf J}$, the Ising measures $\mu_{\bf J,h}$ defined 
in~\eqref{eq:Gibbs} are all stationary, regardless of the choice of~${\bf h}$.
In fact, these are the {\it only\/} stationary distributions, as proved in~\cite{CapSin}: 
\begin{lemma}\label{station} \cite[Lemma 3.2]{CapSin}\ \ 
For both the above dynamics, for any fixed  interaction matrix ${\bf J}$, a distribution $\mu\in\cP(\O)$ is stationary for \eqref{cw2} if and only
if $\mu$ has the form~\eqref{eq:Gibbs} for some choice of the fields~${\bf h}$. 
\end{lemma}

We now address the convergence result claimed in Theorem~\ref{thm:convintro} in the introduction, which we
restate here more formally.

\begin{theorem}\label{th:conv}
Fix an interaction matrix ${\bf J}$. For any initial distribution $p\in\cP(\O)$ with nondegenerate marginals, both the nonlinear block 
dynamics and the nonlinear Glauber dynamics  satisfy the convergence
\begin{align}
\label{peqconver}
\|T_t(p)-\mu_{\bf J,h}\|_{\rm TV}\to 0\qquad\hbox{\rm as $t\to\infty$},
\end{align}
where ${\bf h}$ is the unique choice of external fields such that $\mu_{\bf J,h}$ and $p$ have the same marginals, i.e., $(\mu_{\bf J,h})_x=p_x$ for all $x\in V$. 
\end{theorem}
We will prove Theorem~\ref{th:conv} as a consequence of a more general convergence theorem (Theorem~\ref{th:genconv} below), 
together with the irreducibility property established in Lemma~\ref{lem:irred}. 

Our general convergence theorem is based on the following more abstract framework for mass action kinetics.
Given a finite space of types~$\O$, consider the nonlinear mass action
dynamics $T_t(p)\in\cP(\O)$ defined by $T_t(p) = T_{t-1}(p) \circ T_{t-1}(p)$, $T_0(p)=p$, where the collision 
operator~$\circ$ is defined  as in \eqref{cwnew}, and $\cQ$ is now a generic 
probability kernel on $\O\times\O$, i.e., $\cQ(\si,\si';\t,\t')\geq 0$ for all $\si,\si',\t,\t'\in\O$ and 
\begin{equation}\label{kerQ}
\sum_{\t,\t'\in\O}\cQ(\si,\si';\t,\t')=1\,,\qquad \forall\,\si,\si'\in\O.
\end{equation}
As always we assume the mild exchange symmetry property 
\begin{equation}\label{exchsymQ}
\kernel{\sigma}{\sigma'}{\tau}{\tau'} = \kernel{\sigma'}{\sigma}{\tau'}{\tau}.
\end{equation}
Indeed, as a consequence of~\eqref{exchsymQ}, we may actually assume w.l.o.g.\ that $\cQ$ satisfies the additional
symmetries
\begin{equation}\label{symQ}
\cQ(\si,\si';\t,\t')=\cQ(\si,\si';\t',\t) = \cQ(\si',\si;\t,\t')\,,\qquad \forall\,\si,\si',\t,\t'\in\O.
\end{equation}
To see this, note that because of the symmetry of the factor $p(\si)p(\si')$ in the dynamics~\ref{cwnew},
one can always replace $\cQ$ in~\eqref{cwnew} by the symmetrized kernel 
\begin{equation}
\label{pq22sym}
\bar \cQ(\si,\si';\t,\t') = \frac12 (\cQ(\si,\si';\t,\t')+\cQ(\si',\si;\t,\t'))
\end{equation}
without altering the dynamics.  Using~\eqref{exchsymQ}, we immediately see that this symmetrized kernel
satisfies~\eqref{symQ}.  Accordingly, we will assume~\eqref{symQ} from now on.

We also assume that the kernel has {\em positive diagonal elements}, i.e.,
\begin{equation}\label{posQ}
\cQ(\si,\si';\si,\si')>0\,,\qquad \forall\,\si,\si'\in\O;
\end{equation}
in particular, this rules out periodic behavior.  However, note that, as usual, we do {\it not\/} assume that $\cQ$ is irreducible.
We will also not assume that $\cQ$ is reversible (detailed balanced), as it is in our Ising model systems; rather, it will be 
enough to assume the weaker property 
that there exists a strictly positive distribution $\mu\in\cP_+(\O)$ such that
\begin{equation}\label{piQ}
\sum_{\si,\si'\in\O}\mu(\si)\mu(\si')\cQ(\si,\si';\t,\t')=\mu(\t)\mu(\t')\,,\qquad \forall\,\t,\t'\in\O.
\end{equation}
This is equivalent to saying that the product distribution $\mu\otimes\mu\in\cP(\O\times\O)$ 
is {\it $\cQ$-invariant}, i.e., $(\mu\otimes\mu)\cQ = \mu\otimes\mu$.
If~\eqref{piQ} holds for some positive~$\mu$, we say that $\cQ$ is {\it balanced}\footnote{In the chemical reaction networks literature,
this property is known as ``complex balance".}.
Note that if $\cQ$ is detailed balanced (reversible), as defined in~\eqref{cw2qq},
then it is also balanced.  Moreover, it is easy to check that any~$\mu$ satisfying~\eqref{piQ}
must necessarily also be stationary for the mass-action dynamics defined by~$\cQ$.  Indeed, for any balanced
system, it will follow from the proof of Theorem~\ref{th:genconv} below that the converse is also true: any stationary~$\mu$ must satisfy~~\eqref{piQ} (see Proposition~ \ref{prop:invrev}).

The following general theorem says that, for any mass action kernel satisfying the above properties,
irreducibility of the initial distribution (as specified in Definition~\ref{def:irred}) is sufficient to guarantee
convergence.


\begin{theorem}\label{th:genconv}
Suppose the kernel $\cQ$ satisfies \eqref{symQ} and \eqref{posQ}, and is also balanced.
Then, for any irreducible initial distribution~$p$, we have that  $T_t(p)\to \nu$ as $t\to\infty$ for some stationary 
$\nu\in\cP_+(\O)$.  Moreover, $\nu\otimes \nu$ is $\cQ$-invariant.
\end{theorem}
\noindent
Note that, in general, the limit point~$\nu$ will depend on the initial distribution~$p$.  

The first ingredient  in the proof of Theorem~\ref{th:genconv} is decay of relative entropy. 
Let $\cD(\cdot\Vert\cdot) $ denote the relative entropy, or KL-divergence, for probability measures on $\O$.
I.e., $\cD(p\Vert q) := \sum_{\o\in\O}p(\o)\log\frac{p(\o)}{q(\o)}$. 
\begin{lemma}\label{le:relent}
Suppose the kernel $\cQ$ satisfies \eqref{symQ} and \eqref{posQ}. If $\cQ$ is balanced w.r.t.~$\mu$, then
\begin{align}
\label{pqconver}
\cD(p\circ p\Vert\mu) < \cD(p\Vert \mu),
\end{align} 
for any non-stationary $p\in\cP(\O)$.
\end{lemma}
\begin{proof}
Define $\r := p\otimes p$ and $\pi :=\mu\otimes \mu$.  The assumption that $\cQ$ is balanced w.r.t.~$\mu$
implies that $\pi\cQ=\pi$, while the fact that $p$ is not stationary implies that $\r\cQ\ne\r$.  Now we 
may equivalently write the operation~\eqref{cwnew} in the form
\begin{equation}\label{eq:ent3}
   (p\circ p)(\tau) = \sum_{\si,\si'\!,\tau'} \r(\si,\si') \kernel{\sigma}{\sigma'}{\tau}{\tau'}=\sum_{\tau'}(\r \cQ)(\t,\t').
\end{equation}   
Equation~\eqref{eq:ent3} suggests a two-step decomposition of $p\mapsto p\circ p$.  The first step is 
a mapping on the pair space $\O\times\O$, which takes the product distribution $\r=p\otimes p$ to the
distribution $\r\cQ$ (which is not typically a product).  The second step is the mapping
back down to~$\O$ obtained by marginalizing out the second element~$\tau'$ of the pair.  
We argue that each of these steps separately decreases the relative entropy w.r.t.~$\mu$, the first step yielding a strict decrease whenever $p$ is not stationary.

For the first step, note that $\cQ$ defines one step of a 
Markov chain on $\O\times\O$. Simple convexity considerations imply the inequality $\cD(\r\cQ\Vert\pi) \leq \cD(\r\Vert\pi)$.  
We claim the stronger property that 
\begin{equation}\label{eq:ent10}
\cD(\r\cQ\Vert\pi) < \cD(\r\Vert\pi)
\end{equation} 
whenever $\r\cQ\ne\r$. To prove this, we use the following basic  
fact that applies to any (not necessarily irreducible or reversible) Markov chain with a finite state space, with positive diagonal entries, and with an everywhere positive stationary distribution~$\pi$. Let $\cC_1,\dots,\cC_k$ denote the communicating classes (irreducible components) associated with~$\cQ$. 
Note that the assumption $\pi>0$ implies that there is no transient state, so that the~$\cC_i$ partition the state space $\O\times\O$. Let $F=\r/\pi$ denote the density of $\r$ w.r.t.\ $\pi$ and note that $\r\cQ=\r$ iff $F$ is constant within each $\cC_i$. Thus, to prove \eqref{eq:ent10} it is sufficient to show that  the identity 
\begin{equation}\label{eq:ent11}
\cD(\r\cQ\Vert\pi) = \cD(\r\Vert\pi)\,
\end{equation} 
holds iff $F$ is constant within each $\cC_i$. Direct computation shows that
\begin{equation}\label{eq:ent12}
\cD(\r\Vert\pi)-\cD(\r\cQ\Vert\pi) =  \pi(F\log F) - \pi(\cQ^*F \log \cQ^* F)\,,
\end{equation} 
where $\cQ^*(u,v)=\pi(v)\cQ(v,u)/\pi(u)$ is the time-reversal of $\cQ$. Clearly, if $F$ is constant within each $\cC_i$ then $\cQ^*F =F$ and \eqref{eq:ent11} holds. To see the converse, observe that the strict convexity of $x\mapsto x\log x$ implies that 
\[
H:=\cQ^*(F\log F)-\cQ^*F \log \cQ^* F\geq 0
\] 
and that $H(u)=0$ for some state $u$ if and only if  $F (v')=F(v)$ for all $v,v'$ in the neighborhood of $u$ in the Markov chain graph (i.e., the set of $v$ such that $\cQ(u,v)>0$).  By the assumption~\eqref{posQ} of positive diagonal elements for~$\cQ$, this is equivalent
to $F (v)=F(u)$ for all $v$ such that $\cQ(u,v)>0$. Therefore, $H=0$ everywhere implies that $F$ is constant within each communicating class~$\cC_i$. On the other hand,  it follows from~\eqref{eq:ent12}  and the invariance $\pi \cQ^*=\pi$ that \eqref{eq:ent11} is equivalent to $\pi(H)=0$, which by the  positivity assumption on $\pi$, and the fact that $H\geq 0$, is equivalent to $H=0$ everywhere. This ends the proof of  \eqref{eq:ent10}.

For the second (marginalization) step, we appeal to the well-known fact that, among all distributions on $\O\times\O$
with fixed marginals, the relative entropy w.r.t.\ a product measure $\pi=\mu\otimes\mu$ is minimized when
that distribution is a product distribution. Namely,  if a probability $\nu\in\cP(\O\times\O)$ has marginals $\nu_1,\nu_2$ on the first and second element respectively, one has
 \begin{equation}\label{eq:ent14}
\cD(\nu\Vert\mu_1\otimes\mu_2)\geq \cD(\nu_1\Vert\mu_1)+  \cD(\nu_2\Vert\mu_2)\,,
\end{equation} 
for all probability measures $\mu_1,\mu_2\in\cP_+(\O)$.
To see this, recall the variational principle for relative entropy, asserting that
\[
\cD(\nu\Vert\z) = \sup \{\nu(F) - \log\z(e^F)\},
\] 
where $\nu,\z$ are arbitrary probability measures and $F$ ranges over all real functions. Now, take $\z=\mu_1\otimes\mu_2$ and $F(\si,\si')=\log f_1(\si) + \log f_2(\si')$ where $f_1=\nu_1/\mu_1$ and $f_2=\nu_2/\mu_2$, and observe that 
$\nu(F) = \cD(\nu_1\Vert\mu_1)+  \cD(\nu_2\Vert\mu_2)$ while $\z(e^F) =1 $ because of the product structure.  This proves \eqref{eq:ent14}.

Now, in view of the symmetry \eqref{symQ}, the marginal of $\r\cQ$ on the first element of the pair equals the marginal of $\r\cQ$ on the second element of the pair, and they are both equal to $p\circ p$. 
It follows that 
\begin{equation}\label{eq:ent15}
\cD(\r\cQ\Vert\mu\otimes\mu)\geq 2\cD(p\circ p\Vert \mu)\,.
\end{equation} 
In conclusion, if  $p$ is not stationary we have shown that
\begin{equation}\label{eq:aent15}
   \cD(p\Vert \mu) = \frac12\,\cD(\r\Vert \pi) > \frac12\,\cD(\r\cQ\Vert\pi)
        \geq \cD(p\circ p\Vert \mu).  
        \end{equation}
This completes the proof of \eqref{pqconver} and hence the lemma.
\end{proof} 

 \begin{proof}[Proof of Theorem \ref{th:genconv}]
Let $\cQ$ be balanced w.r.t.~$\mu\in\cP_+(\O)$.
Lemma~\ref{le:relent} shows that $\cD(T_t(p)\Vert \mu)$ is monotonically strictly decreasing with~$t$.
Hence, since $\cD(T_t(p)\Vert \mu)\geq 0$,  we know that $\cD(T_t(p)\Vert \mu)$ converges to a limit, say~$d_\mu$.
By compactness, there exists $\nu\in\cP(\O)$ and a subsequence $t_1,t_2,\ldots$ such that $T_{t_i}(p)\to\nu$ as $i\to\infty$.  
Also, it must be the case that $\cD(T_{t_i}(p)\Vert \mu) \to \cD(\nu\Vert\mu)$, and therefore
\begin{align}
\label{pqconvers}
  d_\mu=\cD(\nu\Vert\mu) =\lim_{t\to\infty} \cD(T_t(p)\Vert \mu)= \inf_{t\in\bbN} \,\cD(T_t(p)\Vert \mu).
\end{align}
Moreover, $\nu$ must be stationary.
To see this, assume for contradiction that it is not.
Then, by  Lemma~\ref{le:relent}, one step of the dynamics must strictly decrease the relative entropy and therefore, for some $\e>0$, 
\begin{equation}\label{eq:conv1}
\cD(T_1(\nu)\Vert \mu) \leq \cD(\nu\Vert \mu) - \e.
\end{equation}
On the other hand, by continuity of the map~$T_1$ and the function~$\cD(\cdot\Vert\mu)$,
since $T_{t_i}(p)\to\nu$, there exists $t_\e\in\bbN$ sufficiently large such that 
\[
  \cD(T_1(T_{t_\e}(p))\Vert \mu) \leq \cD(T_1(\nu)\Vert \mu) +\e/2.
\]
Combining this with~\eqref{eq:conv1} gives $$
\cD(T_{t_\e + 1}(p)\Vert \mu)  \leq \cD(\nu\Vert \mu) - \e/2, $$
which contradicts \eqref{pqconvers}.

Now note that, if $p$ is irreducible, then it must be the case that $\nu\in\cP_+(\O)$. 
Thus we have established that~$\nu$ is a stationary point with full support. Therefore, we may take $\mu=\nu $ in \eqref{pqconvers}, so that $\cD(T_t(p)\Vert \nu)\to \cD(\nu\Vert\nu)=0$ as $t\to\infty$. The latter implies, e.g., by Pinsker's inequality,  that $T_t(p)\to\nu$, completing the proof.
\end{proof}

Before proceeding with the proof of Theorem~\ref{th:conv}, we pause to note the following simple consequences of
the arguments given above.
\begin{proposition}\label{prop:invrev}
Suppose $\cQ$ is a balanced kernel satisfying \eqref{symQ} and \eqref{posQ} and let $p\in\cP(\O)$.
Then $p$ is stationary \,iff\, $p\otimes p$ is $\cQ$-invariant. Moreover, if $\cQ$ is also detailed balanced, then $p$ is stationary \,iff\, $p\otimes p$ is $\cQ$-reversible (i.e.,  iff \eqref{eq:detbal} holds with $\mu$ replaced by $p$). 
\end{proposition}
\begin{proof}
Set $\r=p\otimes p$. The only nontrivial implications are (a)~$p\circ p = p \Rightarrow \r \text{\;is\;} \cQ$-invariant and (b)~if $\cQ$ is detailed balanced, then $p\circ p = p \Rightarrow \r \text{\;is\;} \cQ$-reversible.
The argument in the proof of Lemma~\ref{le:relent} shows that if $p\circ p = p$ then $\r$ is a constant multiple of $\mu\otimes\mu$ over the communicating classes of~$\cQ$. 
This implies~(a). To prove~(b), notice that if the detailed balance condition~\eqref{eq:detbal} holds then it continues to hold if $\mu\otimes\mu$ is replaced by any $\r$ that is a constant multiple of $\mu\otimes\mu$ over the communicating classes of~$\cQ$, since only transitions within such components are relevant in~\eqref{eq:detbal}.
\end{proof}

We conclude this section with a proof of our convergence result for the nonlinear Ising model dynamics, Theorem~\ref{th:conv}.  This will follow immediately
from the general criterion for convergence in Theorem~\ref{th:genconv}, together with the irreducibility property we proved in Lemma~\ref{lem:irred}.

\begin{proof}[Proof of Theorem \ref{th:conv}]
In light of Theorem~\ref{th:genconv}, it suffices to check that both our kernels satisfy all the required properties and that any  $p\in\cP(\O)$ with nondegenerate marginals
is irreducible.  The latter fact, for both nonlinear block and Glauber dynamics, is precisely the statement of Lemma~\ref{lem:irred}. 
The fact that both kernels are balanced follows from the fact that both are reversible.
To check the properties~\eqref{symQ} and~\eqref{posQ}, notice that the diagonal entries are positive for both the nonlinear block dynamics and the nonlinear Glauber dynamics. Moreover, the pair symmetry property \eqref{symQ} is easily seen to be satisfied by the block dynamics kernel~\eqref{pq20}. 
Finally, the single site kernel as written in~\eqref{pq22} does not satisfy~\eqref{symQ}, but it does satisfy the exchange symmetry~\eqref{exchsymQ}
and hence is equivalent to a symmetrized kernel~$\bar\cQ$ that satisfies~\eqref{symQ}, as discussed earlier.
\end{proof}
\begin{remark}\label{rem:bc2}
We proved Lemma~\ref{lem:irred} and Theorem~\ref{th:conv} under the assumption~\eqref{margo} of nondegenerate marginals. 
However, there is no difficulty in extending to the general case of an arbitrary $p\in\cP(\O)$, where some spins may be 
deterministically set to~$+1$ or~$-1$.  It is easy to check that the proofs of Lemmas~\ref{lem:conser} and~\ref{lem:irred}
and Theorem~\ref{th:conv} continue to hold in this setting,
with the external fields at the pinned sites taken to be $\pm\infty$ as discussed in Remark~\ref{rem:bc}.
\end{remark}

\begin{remark}\label{rem:qds}
Theorem~\ref{th:genconv} shows that, for balanced systems, convergence to a stationary point with full support
is guaranteed provided one can prove that the trajectory starting from a given initial distribution~$p$ eventually remains 
uniformly bounded away from zero everywhere; i.e., no type ``dies out".  This observation is already known in the
chemical reaction networks literature~\cite{SM00}.  However, we have provided an alternative proof here for several
reasons: (i)~we are working in discrete rather than continuous time as in the reaction networks community, which necessitates
different arguments; (ii)~we have made extensive use of probabilistic, rather than dynamical systems concepts in our proof;
and (iii)~we aim to make this paper self-contained.  We point out also that our proof follows the same lines as that in~\cite[Theorem~2]{Sinetal}
for the restricted case where $\cQ$ is symmetric (i.e., reversible w.r.t.\ the uniform distribution~$\mu$),
while correcting some omissions in that earlier proof: namely, the assumption of positive diagonal elements~\eqref{posQ}
and the requirement  that the initial condition~$p$ be irreducible.
In light of Theorem~\ref{th:genconv} the key to the proof
of our convergence result  Theorem~\ref{th:conv} for the nonlinear Ising model dynamics is establishing irreducibility, as we do in Lemma~\ref{lem:irred}.
With respect to progress on the Global Attractor Conjecture, the most interesting question here seems to be
that of identifying minimal assumptions on the kernel~$\cQ$ that guarantee such an irreducibility property; we leave this question for future work.
\end{remark}

Theorem \ref{th:conv} provides no quantitative estimate on the {\it rate\/} of convergence to stationarity. 
In particular, there is no explicit dependence on the size of the system~$n$. In analogy with the mixing time analysis
for linear Markov chains, in the remainder of the paper we will study the rate of convergence to equilibrium under the
assumption that the interactions in~${\bf J}$ are sufficiently weak (usually referred to as the ``high temperature" regime).

\section{The nonlinear block dynamics}\label{sec:nonlinearbd}
Let $ T_t(p)$, $t\in\bbN$, denote the evolution of the initial distribution $p\in\cP(\O)$ under the nonlinear block dynamics \eqref{pq20}.  From Theorem \ref{th:conv} we know that for any fixed interaction matrix ${\bf J}$, and any $p\in\cP(\O)$, one has the convergence 
\begin{gather}\label{convergio}
T_t(p)\;\to\; \mu_{\bf J,h}\qquad \hbox{\rm as $t\to\infty$}\,,
\end{gather}
where ${\bf h}$ is the unique vector of external fields such that $ \mu_{\bf J,h}$ and the initial state $p$ have the same marginals at $x$, for all $x\in V$. Our main result for the nonlinear block dynamics (Theorem~\ref{thm:main2} in the introduction) establishes a tight bound on the rate of convergence in \eqref{convergio} as a function of the cardinality $n=|V|$, under the Dobrushin-type high-temperature condition on the interaction matrix~${\bf J}$.  For convenience we restate this result here.
\begin{theorem}
\label{th:mainthla}
There exist absolute constants $\d_0>0$, $c>0$ and $C>0$ such that, if
\begin{gather}\label{Dcondition}
\max_x\sum_{y\in V} |J_{xy}|\leq \d_0\,,
\end{gather}
 then, for any $p\in\cP(\O)$, $t\in\bbN$,
\begin{gather}\label{convergi}
\| T_t(p)-\mu_{\bf J,h}\|_{\rm TV}\leq Cn^2e^{-c\,t},
\end{gather}
where ${\bf h}$ is the unique choice of external fields such that $p_x=(\mu_{\bf J,h})_x$ for all $x\in [n]$. In particular, for any $\e>0$, one has $\|T_t(p)-\mu_{\bf J,h}\|_{\rm TV}\leq \e$ as soon as $t\geq \frac2c\log n + C_1(\e)$, where $C_1(\e)=\frac1c\log(C/\e)$.
\end{theorem}

\subsection{Main ideas of the proof}\label{subsec:sketch}
Before embarking on the details of the proof, we give a high level description of the main steps. 
By symmetry we may rewrite the operator~\eqref{consepq} in the form
\begin{equation}
\label{pq2}
(p\circ q) (\t)= \sum_{\si,\si'}p(\si)q(\si')
\sum_{\L\subseteq V}\g(\L\tc\si,\si')\ind_{\t=\si_\L\si'_{\L^c}},
\end{equation}
where 
\begin{gather}\label{gammala}
\g(\L\tc\si,\si') = \frac{\mu(\si_\L\si'_{\L^c})\mu(\si'_\L\si_{\L^c})}
{\sum_{A\subseteq V}\mu(\si_A\si'_{A^c})\mu(\si'_A\si_{A^c})}\,.
\end{gather}
Thus, for each $\si,\si'\in\O$, $\g(\cdot\tc\si,\si')$ is a probability measure over subsets $\L\subseteq V$. It will be convenient to view the distribution $\g(\cdot\tc\si,\si')$ as a spin system, i.e., a probability measure over spin configurations $\eta\in \{-1,+1\}^n$, by identifying $\eta_x = +1$ with $x\in \L$ and $\eta_x = -1$ with $x\notin \L$. Recall that in the non-interacting case ${\bf J}=0$, the distribution $\g(\cdot\tc\si,\si')$ does not depend on the pair $(\si,\si')$, and is simply the product of Bernoulli measures with parameter $1/2$. 
As described in Section~\ref{sec:fragment},
the dynamics is then entirely governed by 
the pure fragmentation process that starts with the set $V$ and recursively splits sets of vertices uniformly at random until it reaches a collection of singletons. 
The simple argument given in that proof then
allows one to obtain~\eqref{convergi} with $C=\tfrac12$ and $c=\log 2$; see~\cite{Sinetal2,CapPar} for a detailed analysis
of the non-interacting case.

When there is a nontrivial interaction ${\bf J}\neq 0$, this straightforward analysis breaks down. Our proof of Theorem~\ref{th:mainthla} is based on a coupling argument that allows us to reduce the problem to the analysis of a more general process in which the fragmentation mechanism is perturbed by a ``local growth" process arising from the correlations inherent in the interactions. The main idea is that if the local growth is sufficiently sparse, then the underlying fragmentation dominates and eventually the memory of the initial distribution (except for the marginals) is lost. 
 
The first step in the proof is to couple the above
random variable~$\eta$ with distribution $\g(\cdot\tc\si,\si')$ with a random subgraph $G$ of the complete graph $K_n$ having a suitable distribution $\nu$, i.e., we shall write 
\begin{align}
\label{nuga1}
\g(\cdot \tc\si,\si') = \sum_{G}\nu(G)\, \g_{G}(\cdot\tc \si,\si')
\,,
\end{align} 
where the sum extends over all possible subgraphs $G\subseteq K_n$, and $ \g_{G}(\cdot\tc \si,\si')$ is a probability measure on $\O$ for each realization $G$. The key features of this coupling are: 
\begin{itemize}
\item
the distribution $\nu$ does not depend on the pair $(\si,\si')$;
\item   
the distribution $ \g_{G}(\cdot\tc \si,\si')$ depends on the pair $(\si,\si')$ only through the spins \[\si_{V_G}=\{\si_x,\;x\in V_G\}\,,\qquad \si'_{V_G}=\{\si'_x,\;x\in V_G\},\]
where $V_G$ denotes the vertex set of $G$; and
\item  under $\g_{G}(\cdot\tc \si,\si')$, the random variables $\{\eta_y,\; y\in V\setminus V_G\}$ are i.i.d.\ Bernoulli with parameter $1/2$. 
\end{itemize} 
Actually, it will be crucial that $\nu$ can be taken to be the inhomogeneous Erd\H{o}s-R\'enyi random graph with edge weights proportional to \[\l_{xy}:=e^{4|J_{xy}|}-1.\] This ensures that, under the assumption~\eqref{Dcondition}, the graph~$G$ will be sufficiently sparse and the size of the connected components will satisfy good tail bounds.  Note that the expression \eqref{nuga1} can be seen as a form of 
{\em high-temperature expansion}~\cite{VelenikFriedli2017} for the measure $\g(\cdot \tc\si,\si')$. However, a standard high-temperature expansion would produce an expression of the form~\eqref{nuga1} with real-valued coefficients $\nu(G)$ which depend on $(\si,\si')$, while it is crucial for our coupling argument that $\nu$ be a {\em probability measure independent of} $(\si,\si')$.

Armed with the coupling \eqref{nuga1}, we consider all $2^t-1$ interactions in the derivation tree of Section~\ref{sec:fragment}
that produce the final distribution $T_t(p)$. For each interaction we use a realization of the graph $G$ and we specify a realization $B$ of the Bernoulli random variables with parameter $1/2$ which determine $\eta_y$ for $y\in V\setminus V_ G$. We then compute the resulting distribution. Letting $(\vec G,\vec B)=(G_1,B_1),\dots,(G_{2^t-1},B_{2^t-1})$ denote the vector of all such realizations,
we may then write
  \begin{align}
\label{nugas1}
T_t(p) = \sum_{(\vec G,\vec B)} \widehat\nu(\vec G,\vec B) \,T_t (p \tc \vec G,\vec B)
\,,
\end{align} 
where $\widehat\nu$ is a suitable distribution over the realizations $(\vec G,\vec B)$ and $T_t (p \tc \vec G,\vec B)\in\cP(\O)$ represents the distribution at time $t$ conditional on the realizations $(\vec G,\vec B)$. The important point here is that $\widehat\nu $ is independent of the initial conditions, and therefore all correlations in the initial distribution appear only in the measures $T_t (p \tc \vec G,\vec B)$. Moreover, the measure $\widehat\nu $ can naturally be interpreted as a stochastic process that combines fragmentation with local growth. 

The second main ingredient in the proof of Theorem~\ref{th:mainthla} is the identification of an event $\cE_t$ for this process, roughly representing the fact that within time~$t$ all fragments have reached their minimum 
size, and such that 
 \begin{align}
\label{nat11}
 \widehat\nu (\cE_t)\geq 1- A\,n^2\,e^{-b \,t}
\end{align}
for some absolute constants $A,b>0$. The nature of the event $\cE_t$ will be such that 
 \begin{align}
\label{nat01}
T_t (p \tc \vec G,\vec B) = T_t (p' \tc \vec G,\vec B) 
\,,\qquad (\vec G,\vec B)\in\cE_t
\end{align} 
for all $p,p'\in\cP(\O)$ which have the same marginals at every vertex $x\in V$.          Once these facts are established, \eqref{nugas1}, \eqref{nat11} and
\eqref{nat01} imply that for any such $p,p'\in\cP(\O)$ one has
 \begin{gather}\label{converge1}
\|T_t(p)-T_t(p')\|_{\rm TV}\leq  \widehat\nu(\cE_t^c) \max_{(\vec G,\vec B)}\|T_t(p \tc \vec G,\vec B)-T_t(p' \tc \vec G,\vec B)\|_{\rm TV}\leq A\,n^2\,e^{-b \,t} .
\end{gather}
This implies the result of Theorem~\ref{th:mainthla} by taking $p'=\mu_{\bf J,h}$.

We now turn to detailed proofs of the various claims sketched above.

\subsection{Coupling with inhomogeneous Erd\H{o}s-R\'enyi random graphs} \label{sec:coupling}
We start by observing that for every fixed $\si,\si'\in\O$, there is an interaction matrix $\tilde {\bf J} = \tilde {\bf J}(\si,\si')$ such that 
the set of spins to be exchanged, $\g(\cdot\tc\si,\si')$ from \eqref{gammala}, is itself an Ising Gibbs measure~$\mu_{\tilde{\bf J}}$ as defined in~\eqref{eq:Gibbs}.  

\begin{lemma}\label{lem:gamuJ}
For any $\si,\si'\in\O$, we have $\g(\cdot\tc\si,\si')=\mu_{\tilde{\bf J}}$, i.e.,
\begin{gather}\label{gammalaa}
\g(\eta\tc\si,\si')  \propto\exp\left\{\frac12\sum_{x,y\in V} \tilde J_{xy}\eta_x\eta_y \right\}\,,\qquad \eta\in\O,
\end{gather}
where the  interaction matrix $\tilde {\bf J} = \tilde {\bf J}(\si,\si')$ is given by
\begin{gather}\label{jtilde}
\tilde J_{xy} := 2J_{xy}\si_x\si_y\ind_{x\in D(\si,\si')}\ind_{y\in D(\si,\si')}\,,\qquad D(\si,\si') :=\{z\in V: \si_z\neq \si'_z\}.
\end{gather}
\end{lemma}
\begin{proof}
Define 
\[
\xi_x :=\si_x\eta_x
,\qquad \eta_x := \ind_{x\in\L}-\ind_{x\notin\L}.
\]
Note that \[
[\si_\L\si'_{\L^c}]_x=\tfrac12\,\si_x(\eta_x+1) + \tfrac12\,\si'_x(1-\eta_x)\,,
\]
and therefore
\begin{align*}
[\si_\L\si'_{\L^c}]_x[\si_\L\si'_{\L^c}]_y +[\si'_\L\si_{\L^c}]_x[\si'_\L\si_{\L^c}]_y  &=\tfrac12\,\eta_x\eta_y(\si_x-\si'_x)(\si_y-\si'_y)\\&=2\xi_x\xi_y\ind_{x\in D(\si,\si')}\ind_{y\in D(\si,\si')}\,.
\end{align*}
It follows that 
\begin{align*}
&\mu(\si_\L\si'_{\L^c})\mu(\si'_\L\si_{\L^c})=C(\si,\si')\exp\Biggl\{\sum_{x,y\in D(\si,\si')} J_{xy}\xi_x\xi_y\Biggr\}\,,
\end{align*}
where the constant $C(\si,\si')$ does not depend on $\L$. By definition \eqref{gammala}  this concludes the proof. 
\end{proof}

It follows from Lemma \ref{lem:gamuJ} that for any fixed pair $\si,\si'\in\O$, if $\eta\in\O$ is distributed according to $\g(\cdot\tc\si,\si')$, then $\{\eta_x, x\in V\setminus D(\si,\si')\}$ is the Bernoulli measure with parameter $1/2$ and, independently, $\{\si_x\eta_x, x\in D(\si,\si')\}$ is the Ising measure  on $D(\si,\si')$ with zero external fields and interaction $2{\bf J}$.  For our purposes, the problem with this representation of $\g(\cdot\tc\si,\si')$
is that it is structurally highly dependent on the 
configurations $\si,\si'$ through the set $D(\si,\si')$.  
Our goal in this subsection is to formulate an alternative
representation, in Lemma~\ref{lem:decca} below,
that overcomes this problem.

Let $\cG$ be the set of all subgraphs of the complete graph $K_n$ over $V\subseteq [n]$ with isolated vertices removed, and write $\cP(\cG)$ for the set of  probability measures over $\cG$. Thus $G\in\cG$ can be viewed as a collection of unordered pairs $\{x,y\}$ for $x,y\in V$. Note that $G\in\cG$ 
need not be connected and can be the empty graph (with no vertices).  
We write $V_G,E_G$ for the vertex and edge set of $G\in\cG$, respectively.    

\begin{lemma}\label{lem:decca}
Let $\nu_{\bf J}$ be the inhomogeneous Erd\H{o}s-R\'enyi random graph measure associated with the weights $\l_{xy}=e^{4|J_{xy}|}-1$, i.e.,
\begin{gather}\label{aformapp}
\nu_{\bf J}(G) \propto \prod_{\{x,y\}\in E_G} (e^{4|J_{xy}|}-1)\,.
\end{gather}
Then  
\begin{align}
\label{nug1}
\g(\cdot \tc\si,\si') = \sum_{G\in\cG}\nu_{\bf J}(G)\, \mu_{G}(\cdot\tc \si_{V_G},\si'_{V_G})\otimes {\rm Be}_{V \setminus V_G}(\tfrac12)
\,,
\end{align}
where, for any $G\in\cG$, $\mu_{G}(\cdot \tc \si_{V_G},\si'_{V_G})$ is a probability measure on $\{-1,+1\}^{V_G}$ that depends on $\si,\si'$ only through the spins $\si_{V_G},\si'_{V_G}$ and ${\rm Be}_{V \setminus V_G}(\tfrac12)$ is the Bernoulli probability measure on $\{-1,+1\}^{V\setminus V_G}$, which assigns independently the values $\pm1$ with probability $1/2$ to each $x\in V\setminus V_G$. Moreover, the probability measure $\mu_{G}(\cdot\tc \si_{V_G},\si'_{V_G})$ has the product structure 
\begin{align}
\label{prodstr}
\mu_{G}(\cdot\tc \si_{V_G},\si'_{V_G})= \otimes_{i=1}^k\, \mu_{G_i}(\cdot\tc \si_{V_{G_i}},\si'_{V_{G_i}})
\,,
\end{align}
where $G_1,\dots,G_k$ are the connected components of $G$. 
\end{lemma}
\begin{proof}
We start with a high-temperature expansion, which is valid for any probability measure on $\O=\{-1,+1\}^V$ of the form
\begin{gather}\label{forma}
\g_\phi(\eta)=  \frac1Z\exp\Bigl\{\sum_{e} \phi(e)\eta(e) \Bigr\}\,,
\end{gather}
where $e$ denotes an arbitrary undirected edge $e=\{x,y\}$ for $x,y\in V$, $\eta(e)=\eta_x\eta_y$, and $\phi(e)\in\bbR$ are some given interaction coefficients. Note that  \eqref{gammalaa} is of this form. By adding a constant independent of $\eta$ to the exponent, we can rewrite this as
\begin{gather}\label{forma2}
\g_\phi(\eta)=  \frac1{\tilde Z}\prod_e (1+ \d_e(\eta(e))\,,
\end{gather}
where 
\[
\d_e(\eta(e))=\exp\left\{\phi(e)\eta(e) + |\phi(e)| \right\} -1.
\]
The point of adding $ |\phi(e)|$ in the exponent is to ensure that $\d_e(\eta(e))\geq 0$.
Moreover, expanding the product yields
\begin{gather}\label{forma3}
\tilde Z=\sum_{\eta} \prod_e (1+ \d_e(\eta(e)) = \sum_{\eta} \sum_{G\in\cG}\prod_{e\in E_G} \d_e(\eta(e))\,.
\end{gather}
Letting $G_1,\dots,G_k$ denote the (maximal) connected components of $G$, we see that 
\begin{gather}\label{forma4}
\tilde Z=\sum_{G\in\cG}w(G)\,,\qquad w(G) :=2^{|V\setminus V_G|}\prod_{i=1}^k w_c(G_i)\,,
\end{gather}
where, for any connected graph $G\in\cG$, we define 
\begin{gather}\label{wconn}
w_c(G)=\sum_{\eta_{V_G}}\prod_{e\in E_G} \d_e(\eta(e))\,.
\end{gather}
Since $\d_e(\eta(e))\geq 0$, the weights $w(G)$ are all nonnegative. 
Therefore, the measure \eqref{forma} satisfies 
\begin{gather}\label{formap}
\g_\phi(\eta)=  \sum_{G\in\cG} p_\phi(G)\, \hat\mu_G \otimes {\rm Be}_{V \setminus V_G}(\tfrac12),
\end{gather} 
where 
\begin{gather}\label{formap1}
p_\phi(G) :=  \frac{w(G)}{
\sum_{G'\in\cG}w(G')}
\end{gather} 
is a probability measure $p_\phi\in\cP(\cG)$ depending on the interactions  $\phi:=\{\phi(e)\}$, and $\hat\mu_G$ is the probability measure on $\{-1,+1\}^{V_G}$ given by 
the product $\hat\mu_G=\otimes_{i=1}^k v_{G_i}$, where, for any connected graph $G\in\cG$, we define the probability measure 
\begin{gather}\label{wconn2}
v_G(\eta_{V_G}) :=\frac{\prod_{e\in E_G} \d_e(\eta(e))}{w_c(G)}\,
\end{gather}
on $\{-1,+1\}^{V_G}$. 
We remark that, for any $G\in\cG$, $\hat\mu_G$ depends on the interactions 
$\phi_G:=\{\phi(e),\,e\in E_G\}$ only.  We now apply this decomposition to the measure \eqref{gammala}, which by Lemma~\ref{lem:gamuJ} is \eqref{forma} with the choice $\phi(e) = \tilde J_{xy}$, $e=\{x,y\}$, to obtain
\begin{gather}\label{aformappap}
\g(\cdot\tc\si,\si')=  \sum_{G\in\cG} p_{{\bf J},\si,\si'}(G)\, \hat\mu_G(\cdot\tc\si_{V_G},\si'_{V_G}) \otimes {\rm Be}_{V \setminus V_G}(\tfrac12),
\end{gather} 
for a distribution $p_{{\bf J},\si,\si'}\in\cP(\cG)$ that depends on the interactions ${\bf J}=\{J_{xy}\}$ and on the configurations $\si,\si'$ through the set $D(\si,\si')$; see \eqref{jtilde}, \eqref{wconn} and \eqref{formap1}. We note that, because of the dependance on $\si,\si'$, this is not sufficient to prove the desired claim \eqref{nug1}. We shall use a further coupling argument
to lift the decomposition \eqref{aformappap} to the decomposition \eqref{nug1} with the desired properties. 

Observe that the Erd\H{o}s-R\'enyi measure $\nu_{\bf J}$ defined by~\eqref{aformapp}
can be rewritten as in~\eqref{formap1} if we redefine the weights
$w(G)$ as 
\begin{gather}\label{barw}
\bar w(G)= \prod_{e\in E_G} \bar \d_e\,,\qquad \bar \d_e:=(e^{4|J_{xy}|}-1)\,,\;\;e=\{x,y\}.
\end{gather} 
We are going to show that for all $(\si,\si')$, the measure $p_{{\bf J},\si,\si'}$ is stochastically dominated by  $\nu_{\bf J}$, or equivalently that there exists a coupling $\pi\in\cP(\cG\times\cG)$ of $\nu_{\bf J}$ and $p_{{\bf J},\si,\si'}$ such that for all $G,H\in\cG$,
\begin{equation}
\begin{gathered}\label{abformapp}
\sum_{G'\in\cG}\pi(G',H) = \nu_{\bf J}(H)\,,\quad   \sum_{H'\in\cG}\pi(G,H') = p_{{\bf J},\si,\si'}(G)\,, \qquad\hbox{\rm and}\\
\sum_{G,H\in\cG:\,G\subseteq H}\pi(G,H) = 1. 
\end{gathered}
\end{equation}
In particular, it follows that the conditional distribution $\pi(\cdot\tc H)$ of $G\in\cG$ is supported on graphs $G\subseteq H$. Moreover, we will also show that for any $H\in\cG$, 
\begin{itemize}
\item
$\pi(\cdot\tc H)$ depends on the spin configurations $\si,\si'$ only through their values $\si_{V_H},\si'_{V_H}$ on $V_H$,
\item 
$\pi(\cdot\tc H)$ has the product structure $\pi(\cdot\tc H)=\otimes_{i=1}^k \pi(\cdot\tc H_i)$ where $H_1,\dots,H_k$ denote the connected components of $H$.  
\end{itemize}
Once these facts are established we can quickly conclude the proof of Lemma~\ref{lem:decca} as follows. Let $\pi$ be the coupling of $p_{{\bf J},\si,\si'}$ and $\nu_{\bf J}$ as above, so that 
\begin{gather}\label{abformappa}
 p_{{\bf J},\si,\si'}(G) = \sum_{H\in\cG:\, G\subseteq H}\nu_{\bf J}(H) \pi(G\tc H)\,.
\end{gather} 
From \eqref{aformappap} we obtain 
\begin{align}\label{acformapp}
\g(\cdot\tc\si,\si')&=  \sum_{H\in\cG}\nu_{\bf J}(H)\sum_{G\in\cG: \,G\subseteq H} \pi(G\tc H)\, \hat\mu_G(\cdot\tc\si_{V_G},\si'_{V_G}) \otimes {\rm Be}_{V \setminus V_G}(\tfrac12)
\nonumber\\
& = \sum_{H\in\cG}\nu_{\bf J}(H)\, \mu_H(\cdot\tc\si_{V_H},\si'_{V_H}) \otimes {\rm Be}_{V \setminus V_H}(\tfrac12)\,,
\end{align} 
 where we define 
 \begin{align}\label{adformapp}
\mu_H(\cdot\tc\si_{V_H},\si'_{V_H}) :=\sum_{G\in\cG: \,G\subseteq H} \pi(G\tc H)\, \hat\mu_G(\cdot\tc\si_{V_G},\si'_{V_G}) \otimes {\rm Be}_{V_H \setminus V_G}(\tfrac12).
\end{align} 
Since  $\pi(\cdot\tc H)$ depends only on $\si_{V_H},\si'_{V_H}$, the measure in \eqref{adformapp} also depends on $\si,\si'$ only through $\si_H,\si'_H$. Moreover, because of the product structure of $\hat\mu_G(\cdot\tc\si_{V_G},\si'_{V_G}) $ and $\pi(\cdot\tc H)$, the measure $\mu_H(\cdot\tc\si_{V_H},\si'_{V_H})$ defined by \eqref{adformapp} must necessarily also have the desired product structure \begin{align}
\label{prodstr2}
\mu_{H}(\cdot\tc \si_{V_H},\si'_{V_H})= \otimes_{i=1}^k \mu_{H_i}(\cdot\tc \si_{V_{H_i}},\si'_{V_{H_i}})
\end{align}
along the connected components $H_1,\dots,H_k$ of $H$. Therefore, 
the 
decomposition \eqref{acformapp} concludes the proof of the lemma once we establish the desired properties of the coupling~$\pi$. 

To construct the coupling $\pi$ we use the following coupled Glauber-type Markov chains. Let $U_1,U_2,\dots$ denote i.i.d.\ uniform random variables in $[0,1]$ and let $e_1,e_2,\dots$ denote i.i.d.\ uniformly random edges $e=\{x,y\}$ taken  among all $\binom{n}2$ possible choices. For any $G\subseteq H\in\cG$, any fixed $e=\{x,y\}$, let $G^{e,+}$ denote the graph $G\cup\{e\}$, and
let $G^{e,-}$ denote the graph $G\setminus\{e\}$. We define
\begin{align}\label{afformapp}
p(G,e,\pm) = \frac{w(G^{e,\pm})}{w(G^{e,-})+w(G^{e,+})},\qquad \bar p(G,e,\pm) = \frac{\bar w(G^{e,\pm})}{\bar w(G^{e,-})+\bar w(G^{e,+})}
\end{align} 
For any fixed $G_0, H_0\in\cG$, we write $(X_t,Y_t)$, $t=0,1,\dots$, for the Markov chain with $X_0=G_0, Y_0=H_0$, and such that for any $t\geq 1$, 
\begin{enumerate}
\item if $U_t\leq p(X_{t-1},e_t,+)$ then $X_t=X_{t-1}\cup \{e\}$, otherwise $X_t=X_{t-1}\setminus \{e\}$;
\item if $U_t\leq \bar p(Y_{t-1},e_t,+)$ then $Y_t=Y_{t-1}\cup \{e\}$, otherwise $Y_t=Y_{t-1}\setminus \{e\}$.
\end{enumerate}
In words, at each step a uniformly random edge $e_t$ is chosen, and the graphs $X_{t-1},Y_{t-1}$ are updated by adding or removing the edge $e_t$ according to the prescribed probabilities and the common source of randomness~$U_t$. 
This defines the Markov chain with state space $\cG\times\cG$. 
By construction, we have the reversibility conditions
\begin{align*}
&p_{{\bf J},\si,\si'}(G\setminus\{e\}) \,p(G,e,+) = p_{{\bf J},\si,\si'}(G\cup\{e\})\, p(G,e,-)\,;\\
 &\nu_{\bf J}(H\setminus\{e\}) \,\bar p(H,e,+) = \nu_{\bf J}(H\cup\{e\})\,\bar p(H,e,-), 
\end{align*} 
for any $G,H\in\cG$. Therefore, the marginals $X_t$ and $Y_t$,  $t=0,1,\dots$,  are  Markov chains with state space $\cG$ with stationary distributions $p_{{\bf J},\si,\si'}$ and $\nu_{\bf J}$ respectively.  

Let us now observe that the above coupled process preserves ordering, in the sense that if $G_0\subseteq H_0\in\cG$ then $X_t\subseteq Y_t$ for 
all $t=1,2,\dots$.  In view of our construction, to prove this it suffices to note that, for any $G\subseteq H$ and any~$e$,
   \begin{align}\label{mono}
p(G,e,+) \leq \bar p(H,e,+)\,.
\end{align} 
For $e=\{x,y\}$, one has 
 \begin{align}\label{mono2}
p(G,e,+) = \frac{\sum_{\eta_x,\eta_y}\d_e(\eta(e))\sum_{\{\eta_z,z\neq x,y\}}\prod_{e'\in G:e'\neq e}\d_{e'}(\eta({e'}))}{\sum_{\eta_x,\eta_y}[1+\d_{e}(\eta(e))]\sum_{\{\eta_z,z\neq x,y\}}\prod_{e'\in G:e'\neq e}\d_{e'}(\eta(e'))}.
\end{align}
This is an increasing function of the nonnegative variable $\d_e(\eta(e))$ and thus, using 
\[
\d_e(\eta(e))=\exp\left\{\phi(e)\eta(e) + |\phi(e)| \right\}-1
\leq \bar \d_e = e^{4|J_{xy}|}-1,
\] 
where $e=\{x,y\}$, $\eta(e)=\eta_x\eta_y$, and $\phi(e)=\tilde J_{xy}$, one has $p(G,e,+) \leq \hat p(G,e,+) $ where the latter is defined as in \eqref{mono2} with $\d_e(\eta(e))$ replaced by $\bar \d_e$. On the other hand it is not hard to see that   
\begin{align}\label{aafformapp}
\hat p(G,e,+) = \frac{e^{4|J_{xy}|}-1}{e^{4|J_{xy}|}} = 
\bar p(H,e,+)\,,
\end{align} 
for any $G,H\in\cG$. This proves 
the desired monotonicity~\eqref{mono}.

The monotone coupling $\pi$ of $p_{{\bf J},\si,\si'}$ and $\nu_{\bf J}$ may then be defined as the stationary distribution of the Markov chain $(X_t,Y_t)$. The required property \eqref{abformapp} is a consequence of the above construction and the inequality \eqref{mono}, since if we start the Markov chain at $(X_0,Y_0)$ with $X_0\subseteq Y_0$ we will have $X_t\subseteq Y_t$ at all times and thus the limiting distribution $\pi$ as $t\to\infty$ must be supported  on ordered pairs.   

Finally, we need to check that
 for any $H\in\cG$, the conditional distribution $\pi(\cdot\tc H)$ of $G\subseteq H$ depends on the spin configurations $\si,\si'$ only through their values $\si_{V_H},\si'_{V_H}$ on $V_H$.   
To this end, observe that $\pi(\cdot\tc H)$ is the stationary distribution of the Markov chain $(X_t,Y_t)$ obtained by conditioning on $Y_t=H$ for all $t$, or equivalently the Markov chain defined as above with the restriction that any transition that would result in adding or removing an edge from the graph $Y_t=H$ is rejected. This yields the Markov chain $X^H_t$, $t=0,1,\dots$ with state space $\cG_H=\{G\in\cG:\, G\subseteq H\}$ and transition probabilities given by \eqref{mono2} with the restriction that adding an edge $e$ is only allowed if $G\cup\{e\}\subseteq H$.  Since 
the expression  \eqref{mono2} factorizes over connected components of $G$, and since $G\subseteq H$ it follows that  \eqref{mono2} only depends on the coefficients $\phi(e)=\tilde J_{xy}$, where $e=\{x,y\}\in H$, and therefore the stationary distribution $\pi(\cdot\tc H)$ has the same property. This ends the proof of Lemma \ref{lem:decca}.
\end{proof}

So far we have not used the  
weak interaction assumption \eqref{Dcondition}, so Lemma \ref{lem:decca} holds for arbitrary coefficients $J_{xy}$. Next, we observe that the condition~\eqref{Dcondition}
implies a strong sparsity property of the measure $\nu_{\bf J}$  in Lemma \ref{lem:decca}.
From the definition~\eqref{aformapp}, this
measure is the inhomogeneous Erd\H{o}s-R\'enyi random graph where each edge $e=\{x,y\}$, $x,y\in V$, is included independently with probability
 \begin{align}\label{inhomER}
p_{xy} = 1-e^{-4|J_{xy}|}.
\end{align} 
In particular, using $(e^z-1)e^{-z}\leq z$ for $z\ge 0$, one has $p_{xy}\leq 4|J_{xy}|$, and if \eqref{Dcondition} holds, then for any $x\in V$, 
 \begin{align}\label{inhomER2}
\nu_{\bf J}(V_G\ni x) \leq \sum_{y:\,y\neq x} p_{xy}\leq 4\d_0.
\end{align} 
Moreover, if $\d_0$ is sufficiently small, then for any given $x\in V$ the size of the connected component of $G$ at $x$ has an exponential tail.
We make this precise in Lemma~\ref{lem:expmom} below.

\ignore{
\begin{remark}\label{rem:er}
From the definition \eqref{aformapp} one has that
the measure $\nu_{\bf J}$ from Lemma \ref{lem:decca} is the inhomogeneous Erd\H{o}s-R\'enyi random graph where each edge $e=\{x,y\}$, $x,y\in V$, is included independently with probability
 \begin{align}\label{inhomER}
p_{xy} = \frac{e^{4|J_{xy}|}-1}{e^{4|J_{xy}|}}.
\end{align} 
In particular, using $(e^z-1)e^{-z}\leq z$, $z\in[0,\infty)$, one has $p_{xy}\leq 4|J_{xy}|$, and that if \eqref{Dcondition} holds, then for any $x\in V$, 
 \begin{align}\label{inhomER2}
\nu_{\bf J}(V_G\ni x) \leq \sum_{y:\,y\neq x} p_{xy}\leq 4\d_0.
\end{align} 
Moreover, if $\d_0$ is sufficiently small in~\eqref{Dcondition}, then for any given $x\in V$ the size of the connected component of $G$ at $x$ has an exponential tail; see Lemma \ref{lem:tailux} for a detailed estimate.
\end{remark}
}

\subsection{Fragmentation with noise}
We now develop the main construction behind the convergence result in Theorem \ref{th:mainthla}. It is based on a  {\it perturbed\/} fragmentation process, i.e., a process that combines the random fragmentations of the non-interacting case (as described in Section~\ref{sec:fragment}) with some competing noise represented by the random graphs encoding dependencies. 

Given a set $A\subseteq [n]$ and a random graph $G\in\cG$, we define the random set $A'$ as the vertex set of the union of all connected components of $G$ that have non-empty intersection with~$A$. More formally, write $G=\cup_{i=1}^\ell G_i$ where $G_i$ are the connected components of $G$ and let 
\[
G(A)=\bigcup_{i:\;V_{G_i}\cap A\neq \emptyset}\,G_i.
\]
Then we set $A'=V_{G(A)}$. 
We may sample $A'$ starting from $A$ by a breadth-first search, i.e., by revealing sequentially for each $x\in A$ the neighborhood of $x$ in $G$, then recursively the neighborhood of each vertex revealed at the previous step, and so on until there are no more neighbors to reveal. Clearly, $A'$ may contain sites that are not in $A$. However, if $x\in A\cap V_G^c$ then $x\notin A'$. 
If $ A\setminus A'\neq\emptyset$,  for every $x\in A\setminus A'$, we independently declare $x$ to be {\em in} or {\em out} by a fair coin flip. We thus obtain two random sets $A_{\rm in}$ and $A_{\rm out}$, such that $$A_{\rm in}\cup A_{\rm out}=A\setminus A'\,,\qquad A_{\rm in}\cap A_{\rm out}=\emptyset.$$ 
Next, for any $A\subseteq [n]$, we define two sets  
$\Phi_0(A),\Phi_1(A)$ by 
  \begin{align}\label{phi01}
\Phi_0(A) = \Phi_1(A) = \emptyset \,,\qquad \text{if\;\;} |A|\leq 1\,,
\end{align}
and
  \begin{align}\label{phi011}
\Phi_0(A) = A'\cup A_{\rm in}\,,\quad \Phi_1(A) = A'\cup A_{\rm out}
\,,\qquad \text{if\;\;} |A|\geq 2\,.
\end{align}

\begin{definition}\label{def:fragnoise}
The {\em fragmentation plus noise process} $\cF_t$, $t=0,1,\dots$ is the random process defined as follows. For each $t\in\bbN$, $\cF_t$ consists of $2^t$ labeled fragments, i.e., (possibly empty) subsets $F^{(t)}_1,\dots,F^{(t)}_{2^t}$, $F^{(t)}_i\subseteq [n]$, obtained by repeated application of the following rule.  
At time zero we have $\cF_0=[n]$, i.e., $F^{(0)}_1=[n]$. At time $t\in\bbN$, 
if $\cF_{t-1}=(F^{(t-1)}_1,\dots,F^{(t-1)}_{2^{t-1}})$, then for each $i$ independently, we replace $F^{(t-1)}_i$ by $(\Phi_0(F^{(t-1)}_i),\Phi_1(F^{(t-1)}_i))$ where $\Phi_0,\Phi_1$ are the random maps defined by \eqref{phi01}-\eqref{phi011}, so that 
\begin{align}\label{phrag}
\cF_{t}&=(F^{(t)}_1,\dots,F^{(t)}_{2^t})\nonumber \\&=(\Phi_0(F^{(t-1)}_1),\Phi_1(F^{(t-1)}_1),\Phi_0(F^{(t-1)}_2),\Phi_1(F^{(t-1)}_2),\dots,\Phi_0(F^{(t-1)}_{2^{t-1}}),\Phi_1(F^{(t-1)}_{2^{t-1}}))\,. 
\end{align}
We say that the process\/ {\rm dies out} if there is a time $t$ such that all fragments are empty,  i.e., $F^{(t)}_i=\emptyset$ for all $i=1,\dots,2^t$. With slight abuse of notation, we write $\cF_t=\emptyset$ for the latter event. 
\end{definition}

In the non-interacting case ${\bf J}=0$ one has $p_{xy}=0$ for all $\{x,y\}$, and thus $A'=\emptyset$, $\Phi_0(A)=A_{\rm in}, \Phi_1(A)=A_{\rm out}$ and $A\mapsto (\Phi_0(A),\Phi_1(A))$ is one step of a pure fragmentation process, where the set $A$ is partitioned into two subsets using independent fair coin flips for each vertex. In this case $F_i^{(t)}\cap F_j^{(t)}=\emptyset $ for all $i,j=1,\dots,2^t$ and for all $t\in\bbN$. In particular, it is not hard to see that in this case 
\begin{align}\label{dkzero}
\bbP(\cF_t\neq\emptyset)\leq n(n-1)\,2^{-t}\,,\qquad t=1,2,\dots
\end{align}
Indeed, by construction $\cF_t\neq\emptyset$ implies that there exist two vertices $x,y\in[n]$ that belong to the same fragment up to time $t-1$. Thus \eqref{dkzero} follows as in the analysis of Section~\ref{sec:fragment}.
In the interacting case, there is a first stage where the set $A$ grows according to the local branching at  every $x\in A$, and the fragmentation occurs only on those vertices that have an empty neighborhood in $G$.  
Our main technical result below establishes that the fragmentation plus noise process also satisfies a bound of the form \eqref{dkzero}, with a slightly weaker exponential decay rate, provided  \eqref{Dcondition} holds for a suitably small $\d_0>0$.

\begin{lemma}\label{fraglemma}
For any $\d\in(0,1)$, there exists $\d_0>0$, and a constant $C_\d>0$ such that if  \eqref{Dcondition} holds with constant $\d_0$ then
\begin{align}\label{dkz}
\bbP(\cF_t\neq\emptyset)\leq n^2C_\d\,(2-\d)^{-t}\,,\qquad t=1,2,\dots
\end{align}
\end{lemma}
The proof of this lemma is quite technical and is postponed to Section \ref{sec:pflem38}. We turn first to the proof of Theorem~\ref{th:mainthla}.

\subsection{Proof of Theorem \ref{th:mainthla}}
We now have the tools to conclude the proof of Theorem \ref{th:mainthla}, our main result for nonlinear block dynamics.
Recall again the construction of $T_t(p)$ in terms of the binary derivation tree in Section~\ref{sec:fragment}.

By the invariance property \eqref{eq:inv}, the target measure $\mu_{\bf J,h}$ can be obtained at the root by taking the distribution $\mu_{\bf J,h}$ on each leaf. Thus, Theorem \ref{th:mainthla} says that when each leaf is given a distribution $p$ with the same marginals as $\mu_{\bf J,h}$, the two distributions $ T_t(p)$ and $\mu_{\bf J,h}$ can be coupled with an error at most $Cn^2e^{-c t}$ for any~$t$. We shall actually prove the following stronger result. Let $\vec p=(p_1,\dots,p_{2^t})$ and $\vec q=(q_1,\dots,q_{2^t})$ be arbitrary vectors of distributions in~$\cP(\O)$ whose marginals on $\si_x$ at all sites~$x$ agree, i.e., $p_i,q_i\in\cP(\O)$ satisfy 
\begin{gather}\label{samemarg}
(p_i)_x=(p_j)_x=(q_i)_x=(q_j)_x\,,\qquad\;i,j=1,\dots,2^t,\;x\in V. 
\end{gather}
Let $T_t(\vec p)$ (resp., $T_t(\vec q)$) denote the distribution at the root of the binary tree of depth $t$, 
where the leaf labeled $i=1,\dots,2^t$ is equipped with the distribution $p_i$ (resp., $q_i$); recall Figure~\ref{fig1} for the case $t=2$.
\begin{theorem}\label{th:mainarb}
There exist absolute constants $\d_0>0$, $c>0$ and $C>0$ such that, if
\eqref{Dcondition} holds with constant~$\d_0$ then for any choice of initial distributions $\vec p,\vec q$ as in \eqref{samemarg}, 
\begin{gather}\label{convergiarb}
\| T_t(\vec p)- T_t(\vec q)\|_{\rm TV}\leq Cn^2e^{-c\,t},\qquad t\in\bbN.
\end{gather}
\end{theorem}   
Clearly, Theorem \ref{th:mainarb} implies Theorem \ref{th:mainthla} since we may take $p_i\equiv p\in\cP(\O)$ and $q_i\equiv \mu_{\bf J,h}$, where the external fields~$\bf h$ are chosen in such a way that $p$ and $  \mu_{\bf J,h}$ have the same marginals.   

We shall prove Theorem \ref{th:mainarb} by analyzing the interaction history backwards in time, i.e., from the root to the leaves. This is  reminiscent of the coupling from the past approach for linear Markov chains \cite{propp1996exact,LPW}, and to some extent our proof is inspired by ideas that have been developed in that context. In particular, our proof is related to the information percolation framework developed by Lubetzky and Sly in \cite{lubetzky2017universality}.  

Each internal node of the tree is associated with an interaction, or collision, which according to \eqref{pq2} is specified by the random set $\L$ with distribution $\g(\cdot\tc \si,\si')$. For each such interaction we reveal the realization of the graph $G$ and of the Bernoulli variables $B$ in $V\setminus V_G$ that are used in sampling $\Lambda$; see  Lemma \ref{lem:decca}. 
In this way, starting from the root, we have a pair $(G,B)$, $G\in\cG$ and a subset $B\subseteq V\setminus V_G$ is identified with the set 
of $x\in V\setminus V_G$ for which $\eta_x=+1$. Suppose  the descendants of the root have distributions $p$ and $q$ respectively, as in Figure \ref{fig2}. Then, according to Lemma \ref{lem:decca},  the distribution at the root is given by
\begin{align}\label{eq:pqtas}
p\circ q=\sum_{(G,B)} \nu(G,B)\,  T(p,q\tc G,B)\,,
\end{align}
where $\nu(G,B)=2^{-|V\setminus V_G|}\nu_{\bf J}(G)$, and             
for each realization $(G,B)$, 
$ T(p,q\tc G,B)\in\cP(\O)$ is the distribution
\begin{align}\label{eq:pqt}
 T(p,q\tc G,B)(\t)&=\sum_{\si,\si'} p(\si)q(\si')\sum_{\eta_{V_G}}\mu_{G}(\eta_{V_G}\tc \si_{V_G},\si'_{V_G})\,.
\,\ind_{\t\sim (\si,\si',\eta_{V_G},B)}.
\end{align}
Here, for $\t\in\O$, the notation $\t\sim (\si,\si',\eta_{V_G},B)$ is shorthand for 
\begin{align}\label{eq:pqt2}
\t\sim (\si,\si',\eta_{V_G},B)\;\Leftrightarrow\;\begin{cases}\t_x=\si_x\,,& x\in V:\,\eta_x=+1;\\
\t_x=\si'_x\,,& x\in V:\,\eta_x=-1,
\end{cases}\end{align}
with the understanding that the value of $\eta$ on $ V\setminus V_G$ is specified by $B\subseteq  V\setminus V_G$, i.e., $\eta_y=+1$ for $y\in B$ and $\eta_y=-1$ for $y\notin B$.

For every $G\in\cG$, $\mu_{G}(\cdot\tc \si_{V_G},\si'_{V_G})
$ depends on $(\si,\si')$ only through $(\si_{V_G},\si'_{V_G})$, and for every $\eta_{V_G}\in\{-1,1\}^{V_G}$, $\t\in\{-1,1\}^n$ and $B\subseteq V\setminus V_G$, the condition \eqref{eq:pqt2} 
depends on $(\si,\si')$ only through
\[
\{\si_x,\,\,x\in V_G \cup B\}\,,\qquad \{\si'_x,\,\,x\in V_G \cup ((V\setminus V_G)\setminus B)\}\,.
\]
Therefore, to compute the distribution $ T(p,q\tc G,B)$ we only need the marginals $p_{V_G\cup B}$ and $q_{V_G\cup B'}$, where $B' =  (V\setminus V_G)\setminus B$. 
Note that we may identify $V_G\cup B$ with the set $\Phi_0([n])$ and $V_G\cup B'$ with the set $\Phi_1([n])$, where $\Phi_0,\Phi_1$ are the maps defined in \eqref{phi011}. Indeed, by definition of the measure $\nu_{\bf J}$, these random sets have the same distribution since, when $A=[n]$,  $V_G$ is equivalent to $A'$ and $B$ is equivalent to $A_{\rm in}$. One way to rephrase this is to say that, as far as the distribution $ T(p,q\tc G,B)$ is concerned, the only relevant information about the distribution $p$ is contained in the set $\Phi_0([n])$ and the only relevant information about the distribution $q$ is contained in the set $\Phi_1([n])$.

\begin{figure}[h]
\center

\begin{tikzpicture}[scale=0.8]
    
    \draw  (1,-1) -- (2,1);
    \draw  (3,-1) -- (2,1);
   \draw  (2,1) -- (4,3);
    \draw  (5,-1) -- (6,1);
    \draw  (7,-1) -- (6,1);
    \draw  (6,1) -- (4,3);
    
    \node[shape=circle, draw=black, fill = white, scale = 1.5]  at (4,3) {}; 
    \node[shape=circle, draw=black, fill = white, scale = 1.5]  at (2,1) {};
    \node[shape=circle, draw=black, fill = gray, scale = 1.5]  at (1,-1) {};
    \node[shape=circle, draw=black, fill = gray, scale = 1.5]  at (3,-1) {};
    \node[shape=circle, draw=black, fill = white, scale = 1.5]  at (6,1) {};
    \node[shape=circle, draw=black, fill = gray, scale = 1.5]  at (5,-1) {};
    \node[shape=circle, draw=black, fill = gray, scale = 1.5]  at (7,-1) {};
    
    \node at (.25,-1) {$\;p'$};
    \node at (2.25,-1) {$\;q'$};
    \node at (4.25,-1) {$\;p''$};
    \node at (6.25,-1) {$\;q''$};
    \node at (1,.8) {$p' \circ q'\;\;$};
    \node at (3.2,1.1) {$(G',B')$};
\node at (5,.8) {$p'' \circ q''\;\;$};
    \node at (7.3,1.1) {$(G'',B'')$};

    \node at (3,2.9) {$p \circ q$};
\node at (5.2,3.1) {$(G,B)$};

\end{tikzpicture}
\caption{Graphical representation of the distribution at the root,  when $t=2$ and the leaves are equipped with distributions $p',q',p'',q''$. Each internal node is equipped with a realization of the random pair $(G,B)$, where $G\in\cG$ and $B\subseteq V\setminus V_G$. }
\label{fig2}
\end{figure}
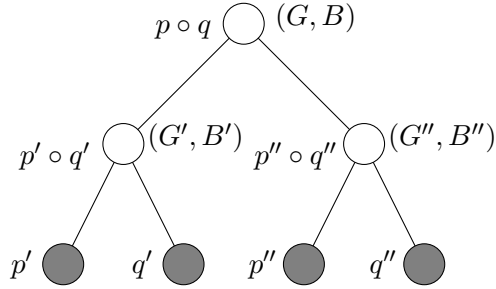

Next, we move one step backwards in time and consider the interaction that produced the distribution $p$ from the previous computation. Suppose that $p',q'$ are the distributions at the two descendants of~$p$ respectively, so that $p=p'\circ q'$, as in Figure~\ref{fig2}. Suppose we revealed the realization $(G',B')$ of the graph and Bernoulli variables associated with this interaction. Note that we can use the same expressions \eqref{eq:pqt}-\eqref{eq:pqt2} to compute $ T(p',q'\tc G',B')$, provided we replace $(p,q)$ by $(p',q')$ and $(G,B)$ by $(G',B')$. However, the key point is that we now only need the marginal of $p$ on the set $\Phi_0([n])$, and therefore when we compute $ T(p',q'\tc G',B')(\t)$ as above we can sum away all $\t_y$, $y\notin \Phi_0([n])$, so that the indicator function \eqref{eq:pqt2} is relevant only for sites $x\in \Phi_0([n])$, and the distribution of $\eta_x$, $x\in V_{G'}\cap \Phi_0([n])$, under $\mu_{G'}(\cdot\tc \si_{V_{G'}},\si'_{V_{G'}})$
 is only influenced by the spins $\si_y,\si'_y$ for $y\in G'(\Phi_0([n]))$, where we recall that $G'(A)$ is the 
 union of all connected components of $G'$ that have nonempty intersection with the set~$A$. The latter property is a consequence of the product structure of the measure $\mu_{G'}(\cdot\tc \si_{V_{G'}},\si'_{V_{G'}})$; see Lemma~\ref{lem:decca}. 
 Hence we can neglect all connected components of $G'$  that do not intersect $\Phi_0([n])$, and we can discard the information about all Bernoulli variables at sites $y\in V\setminus V_{G'}$ such that $y\notin \Phi_0([n])$. A close inspection of our definition of the maps $\Phi_0$ and $\Phi_1$ then reveals that the only information about the distributions $p',q'$ that is needed to compute $ T(p',q'\tc G',B')$ is contained in the marginals $p'_{\Phi_0(\Phi_0([n]))}$ and $q'_{\Phi_1(\Phi_0([n]))}$. 
 
 Similarly, considering the interaction which produced the distribution $q=p''\circ q''$ (see Figure \ref{fig2}),  we may fix a realization $(G'',B'')$ of the graph and Bernoulli variables, and repeating the above reasoning one has that the only information about the distributions $p'',q''$ that is needed to compute $ T(p'',q''\tc G'',B'')$ is contained in the marginals $p''_{\Phi_0(\Phi_1([n]))}$ and $q''_{\Phi_1(\Phi_1([n]))}$. We note that, after two steps of the evolution, conditional on the realizations of the variables $(G,B),(G',B'),(G'',B'')$, we have obtained a probability measure at the root depending only on $(G,B),(G',B'),(G'',B'')$ and on the marginals \[p'_{\Phi_0(\Phi_0([n])))},q'_{\Phi_1(\Phi_0([n]))},p''_{\Phi_0(\Phi_1([n]))},q''_{\Phi_1(\Phi_1([n]))}.\]
Equivalently, using the notation introduced in Definition \ref{def:fragnoise}, after two steps we have that the correlations of the initial distributions at the leaves are entirely contained in the fragments
\[
\cF_2=(F^{(2)}_1,F^{(2)}_2,F^{(2)}_3,F^{(2)}_4) = (\Phi_0(\Phi_0([n])),\Phi_1(\Phi_0([n])),\Phi_0(\Phi_1([n])),\Phi_1(\Phi_1([n]))).
\]

\begin{example}\label{ex:n4} 
Consider the following simple example with $n=4$. 
Suppose $G=\{\{1,2\}\}$, i.e., $G$ consists of the single edge $\{1,2\}$ and suppose that $B=\emptyset$, i.e., both $3$ and $4$ are {\em out}. This gives $F^{(1)}_1=\Phi_0([n]) = \{1,2\}$, $F^{(1)}_2=\Phi_1([n]) =\{1,2,3,4\}$.
Suppose also that $G'=\{\{3,4\}\}$, and $B'=\{1\}$. Thus $F^{(2)}_1 = \{1\}$, and $F^{(2)}_2=\{2\}$. Suppose finally that $G''=\emptyset$, $B''=\{1,2\}$. Then $F^{(2)}_3 = \{1,2\}$, and $F^{(2)}_4=\{3,4\}$. Thus, one has the following fragmentation after two steps
\begin{gather}
\cF_0=\{1,2,3,4\}\,,\\
\cF_1=(\{1,2\},\{1,2,3,4\})\,,\\
\cF_2=(\{1\},\{2\},\{1,2\},\{3,4\}) .
\end{gather}
Note that in this example, conditionally on the given realizations of the variables $(G,B),(G',B')$, and $(G'',B'')$, the distribution  $p\circ q$ at the root in Figure \ref{fig2} can be computed only using the marginals \[(p')_{\{1\}},(q')_{\{2\}},(p'')_{\{1,2\}},(q'')_{\{3,4\}}\] of the input distributions $p',q',p'',q''$.
\end{example}
Repeating the above procedure one has that, after $t$ steps, conditional on the realizations of all graphs and Bernoulli variables involved in each of the $2^t-1$ interactions, the only information needed from leaf~$i$ is contained in the marginal of the distribution at that leaf on the fragment~$F^{(t)}_i$, for each $i=1,\dots,2^t$, as defined in   
Definition~\ref{def:fragnoise}. The crucial observation is that, as soon as a fragment either becomes empty or contains one site only, then the information carried by the corresponding leaf is irrelevant. Indeed, if it is empty this is obvious, while if it contains one site only then the marginal at that site is irrelevant since all marginals are assumed to be fixed, and in particular they are the same in any choice of the initial conditions~$\vec p$ or~$\vec q$. This explains why we introduced the killing step~\eqref{phi01} in our definition of the fragmentation plus noise process~$\cF_t$, which in turn is crucial to the probability of extinction we are able to establish in Lemma~\ref{fraglemma}.

The above discussion shows that, if we denote by 
\[
(\vec G,\vec B)=(G^{(1)},B^{(1)},\dots,G^{(2^t-1)},B^{(2^t-1)})
\] 
the vector of realizations of the random graphs and Bernoulli variables involved in each interaction at the $2^t-1$ internal nodes of the binary tree of depth $t$, one can write
\begin{gather}\label{eqcq1}
 T_t(\vec p)=\sum_{(\vec G,\vec B)} \widehat\nu(\vec G,\vec B)\,  T_t(\vec p\tc \vec G,\vec B)
\end{gather}
 where 
 \begin{gather}\label{eqcq2a1}
 \widehat\nu(\vec G,\vec B)=\prod_{i=1}^{2^t-1}
 \nu_{\bf J}(G^{(i)})\,2^{-|V\setminus V_{G^{(i)}}|}
 \end{gather}
  is the distribution of the random graphs and Bernoulli variables, while $ T_t(\vec p\tc \vec G,\vec B)$ is some probability measure that may depend on $(\vec G,\vec B)$ and on $\vec p=(p_1,\dots,p_{2^t})$ in a complicated way but has the property that its dependence on the distribution $p_i$ from the $i$-th leaf occurs only through the marginal of $p_i$ on the fragment $F^{(t)}_i$. In particular, if $\cF_t=\emptyset$, i.e., $F^{(t)}_i = \emptyset$ for all $i$, then 
 $ T_t(\vec p\tc \vec G,\vec B)= T_t(\vec q\tc \vec G,\vec B)$. We note that the event $\cF_t=\emptyset$ is measurable with respect to $(\vec G,\vec B)$, so that we may write
 \begin{gather}\label{eqcq2}
 T_t(\vec p\tc \vec G,\vec B)= T_t(\vec q\tc \vec G,\vec B)\,,\qquad (\vec G,\vec B)\in \{\cF_t=\emptyset\}\,,
\end{gather}
and hence
\begin{gather}\label{convergiarba}
\| T_t(\vec p)- T_t(\vec q)\|_{\rm TV}\leq 
\sum_{(\vec G,\vec B)\notin\{\cF_t=\emptyset\}} \widehat\nu(\vec G,\vec B)\,.
\end{gather}
Next, we note that 
\begin{gather}\label{convergiarbat}
\sum_{(\vec G,\vec B)\notin\{\cF_t=\emptyset\}} \widehat\nu(\vec G,\vec B)
=\bbP(\cF_t\neq\emptyset)\,,\qquad t\in\bbN\,,
\end{gather}
where the latter is the probability estimated in Lemma \ref{fraglemma}. Indeed, \eqref{convergiarbat} is a consequence of our definition of the fragmentation plus noise  process: we have already observed that each step of the fragmentation 
\[
F^{(t-1)}_i\longrightarrow (\Phi_0(F^{(t-1)}_i),\Phi_1(F^{(t-1)}_i))
\]
is produced with the correct distribution,   and  the product structure \eqref{eqcq2a1} of the measure $\widehat\nu(\vec G,\vec B)$ guarantees that all such steps are performed independently. From Lemma \ref{fraglemma} we thus conclude that,
for any $\d\in(0,1)$, there exists a constant $\d_0>0$ in~\eqref{Dcondition}, and a constant $C_\d>0$ such that 
\begin{align}\label{dkzo}
\| T_t(\vec p)- T_t(\vec q)\|_{\rm TV}\leq 
 n^2C_\d\,(2-\d)^{-t}\,,\qquad t=1,2,\dots
\end{align}
This implies \eqref{convergiarb} (and in fact shows that we can take the constant $c$ as close as we wish to $\log 2$ provided $\d_0$ is taken suitably small).
This ends the proof of Theorem \ref{th:mainarb} and Theorem \ref{th:mainthla}. \hfill $\square$

We now need to provide the missing proof of Lemma~\ref{fraglemma}.

\subsection{Proof of Lemma \ref{fraglemma}}\label{sec:pflem38}
The event $\cF_t\neq\emptyset$ implies that there exists a fragment $F^{(t-1)}_i$ at time $t-1$ that has cardinality $|F^{(t-1)}_i|\geq 2$. Indeed, by construction, if all fragments 
have size at most $1$ at time $t-1$, then $\cF_t=\emptyset$; see \eqref{phi01}. Since there are $2^{t-1}$ fragments at time $t$, and since all the $F^{(t-1)}_i$ have the same distribution, by a union bound it suffices to show that for any $\d\in(0,1)$, 
\begin{align}\label{dkz1}
\bbP(|F^{(t-1)}_1|\geq 2)\leq n^2C_\d\,(4-\d)^{-t}\,,\qquad t=1,2,\dots\,,
\end{align}
provided \eqref{Dcondition} holds with a sufficiently small $\d_0>0$.
 To prove \eqref{dkz1} we shall use a 
 stochastic domination argument that bounds the evolution of the fragment $F^{(t-1)}_1$, $t\geq 1$ by means of independent {\em labeled branching processes}.

 Consider 
  $n$ independent  processes $X^y:=\{X^y_\ell, \;\ell=0,1,\dots\}$, $y\in[n]$, such that for each $y\in[n]$, $X^y$ is the labeled branching process with $X^y_0=\{y\}$ and such that, at time $t\in\bbN$, each individual with label~$x$ in the $(t-1)$-th generation independently gives birth to the set of individuals $U\subseteq [n]$ with offspring distribution 
  \begin{align}\label{offspringx}
 \mu_x (U) = \begin{cases}\tfrac12(1- \r_x)&U=\emptyset\;\;\text{or}\;\; U=\{x\};\\
 \sum_{G\in\cG}\nu_{\bf J}(G)\,\ind_{G(x)=U} & |U|\geq 2,
 \end{cases}
 \end{align}
where $G(x)$ denotes the connected component of $G$ containing~$x$ and, 
for any $x\in [n]$,
\begin{align}\label{offspringa}
\r_x:= 1-\prod_{z\in[n]\setminus\{x\} }(1-p_{xz}) = \sum_{G\in\cG}\nu_{\bf J}(G)\,\ind_{G(x)\neq\emptyset}
 \end{align}
is the probability that $x$ has a non-empty neighborhood in the random graph defined by \eqref{inhomER}. 
Notice that by definition either $G(x)$ is empty or $|G(x)|\geq 2$, and therefore \eqref{offspringx}, for any $x\in[n]$, defines a probability measure on subsets of $[n]$: a sample $U$ from $\mu_x$ is obtained by first sampling the neighborhood $G(x)$ from $\nu_{\bf J}$; if $|G(x)|\geq 2$ then we set $U=G(x)$; if $G(x)=\emptyset$ then we flip a fair coin and set $U=\emptyset$ if heads and $U=\{x\}$ if tails.

Let $N^y(t)$
denote the size of the whole population of the labeled branching process $X^y$ at time $t-1$, i.e., the total number of individuals generated up to time $t-1$. The proof of Lemma \ref{fraglemma} is based on the following bound on the exponential moment of the random variable~$X^y(t)$. 

\begin{lemma}\label{lem:expmom}
For all $a\in(0,1)$, there exists $\d_0>0$ and $C_a>0$ such that if \eqref{Dcondition} holds with constant~$\d_0$ then,
for all $y\in[n]$ and all $t\in\bbN$, 
\begin{align}\label{expmom}
\bbE[2^{aN^y(t)}] \leq C_a \,.
\end{align}
\end{lemma}
We postpone the proof of Lemma \ref{lem:expmom} and conclude the proof of Lemma~\ref{fraglemma} assuming the validity of this bound. 

Now denote by $X^y_{\ell}$ the set of all individuals generated at the $\ell$-th step, i.e., the $\ell$-th generation of the process $X^y$, and let $|X^y_{\ell}|$ denote its cardinality. 
With this notation, an inspection of the definition of the fragmentation plus noise process shows that
$ F^{(t-1)}_1$ is stochastically dominated by the union of the $X^y$'s, i.e., $\{F^{(\ell-1)}_1,\,\ell\in\bbN\}$ and the independent processes $\{X^y,\,y\in[n]\}$ can be coupled so that, for any $\ell=1,2,\dots$,
\begin{align}\label{F1t-1}
 |F^{(\ell-1)}_1|\leq 
 \sum_{y\in[n]}|X^y_{\ell-1}|
\,. \end{align}
The main observation at this point is that, from the definition of fragmentation plus noise, 
the event $|F^{(t-1)}_1|\geq 2$ implies that $|F^{(\ell-1)}_1|\geq 2$ for all $1\leq \ell\leq t$. 
This in turn, by the domination \eqref{F1t-1}, implies that at all times $1\leq \ell\leq t-1$ one has 
\[
\sum_{y\in[n]}|X^y_{\ell}|
\geq 2.
\] 

For this to happen 
   there must be two processes $X^y,X^z$ and a time $s<t$ such that both $X^y,X^z$ are alive up to time $s$ and such that $X^y$ has at least two individuals in each generation from time $s+1$ to time $t-1$. For example, taking $s=t-1$, this includes the case where both $X^y,X^z$ are alive up to time $t-1$, while taking $s=0$ it includes the case where all processes die at the first time step, except for $X^y$ which has $|X^y_\ell|\geq 2$ for all $1\leq \ell\leq t-1$. We refer to Figure~\ref{fig12} for an illustrative example. 

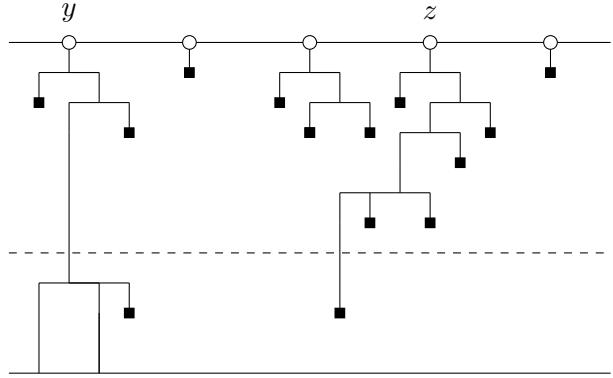
\begin{figure}[h]
\center

\begin{tikzpicture}[scale=0.8]
    
    \draw  (0,1.5) -- (10,1.5);
    \draw  (0,7) -- (10,7);
    
    \draw  [dashed,black] (0,3.5) -- (10,3.5);

\draw  (3,6.9) -- (3,6.5);
    \draw  (9,6.9) -- (9,6.5);

    \draw  (5,6.9) -- (5,6.5);
     \draw  (5,6.5) -- (5.5,6.5);
     \draw  (5,6.5) -- (4.5,6.5);
\draw  (4.5,6.5) -- (4.5,6);
\draw  (5,6.5) -- (5.5,6.5);
\draw  (5.5,6.5) -- (5.5,6);
\draw  (5.5,6) -- (5,6);
\draw  (6,6) -- (5.5,6);
\draw  (5,5.5) -- (5,6);
\draw  (6,5.5) -- (6,6);

 \draw  (7,6.9) -- (7,6.5);
     \draw  (7,6.5) -- (7.5,6.5);
     \draw  (7,6.5) -- (6.5,6.5);
\draw  (6.5,6.5) -- (6.5,6);
\draw  (7,6.5) -- (7.5,6.5);
\draw  (7.5,6.5) -- (7.5,6);
\draw  (7.5,6) -- (7,6);
\draw  (8,6) -- (7.5,6);
\draw  (7,5.5) -- (7,6);
\draw  (8,5.5) -- (8,6);

\draw  (7,5.5) -- (7.5,5.5);
      \draw  (7,5.5) -- (6.5,5.5);
\draw  (6.5,5.5) -- (6.5,5);
\draw  (7.5,5.5) -- (7.5,5);
\draw  (6.5,5) -- (6.5,4.5);
\draw  (6.5,4.5) -- (5.5,4.5);
\draw  (7,4.5) -- (6.5,4.5);
\draw  (7,4.5) -- (7,4);
\draw  (5.5,4.5) -- (5.5,4);
\draw  (6,4.5) -- (6,4);
\draw  (5.5,4) -- (5.5,2.5);

  \draw  (1,6.9) -- (1,6.5);
     \draw  (1,6.5) -- (1.5,6.5);
     \draw  (1,6.5) -- (0.5,6.5);
\draw  (0.5,6.5) -- (0.5,6);
\draw  (1,6.5) -- (1.5,6.5);
\draw  (1.5,6.5) -- (1.5,6);
\draw  (1.5,6) -- (1,6);
\draw  (2,6) -- (1.5,6);
\draw  (1,5.5) -- (1,6);
\draw  (2,5.5) -- (2,6);
\draw  (1,5.5) -- (1,3);

     \draw  (1,3) -- (1.5,3);
     \draw  (1,3) -- (0.5,3);
\draw  (0.5,3) -- (0.5,2.5);
\draw  (1,3) -- (1.5,3);
\draw  (1.5,3) -- (1.5,2.5);
\draw  (0.5,1.5) -- (0.5,2.5);
\draw  (1.5,1.5) -- (1.5,2.5);
\draw  (1.5,1.5) -- (1.5,2.5);

\draw  (1.5,3) -- (2,3);
\draw  (2,3) -- (2,2.5);

    \node[shape=circle, draw=black, fill = white, scale = .5]  at (1,7) {}; 
     \node[shape=circle, draw=black, fill = white, scale = .5]  at (3,7) {};
      \node[shape=circle, draw=black, fill = white, scale = .5]  at (5,7) {};
       \node[shape=circle, draw=black, fill = white, scale = .5]  at (7,7) {};
        \node[shape=circle, draw=black, fill = white, scale = .5]  at (9,7) {};  
           
    \node[shape=rectangle, draw=black, fill = black, scale = .5]  at (3,6.5) {};   
   \node[shape=rectangle, draw=black, fill = black, scale = .5]  at (9,6.5) {};   
    \node[shape=rectangle, draw=black, fill = black, scale = .5]  at (7.5,5) {};   
 \node[shape=rectangle, draw=black, fill = black, scale = .5]  at (6,4) {};  
  \node[shape=rectangle, draw=black, fill = black, scale = .5]  at (7,4) {};  
  \node[shape=rectangle, draw=black, fill = black, scale = .5]  at (2,2.5) {};  
   
    \node[shape=rectangle, draw=black, fill = black, scale = .5]  at (4.5,6) {};  
      \node[shape=rectangle, draw=black, fill = black, scale = .5]  at (6,5.5) {};   
      \node[shape=rectangle, draw=black, fill = black, scale = .5]  at (5,5.5) {};   
      \node[shape=rectangle, draw=black, fill = black, scale = .5]  at (5.5,2.5) {};   
      \node[shape=rectangle, draw=black, fill = black, scale = .5]  at (2,5.5) {};   
      \node[shape=rectangle, draw=black, fill = black, scale = .5]  at (0.5,6) {};   
         
        \node[shape=rectangle, draw=black, fill = black, scale = .5]  at (6.5,6) {};  
      \node[shape=rectangle, draw=black, fill = black, scale = .5]  at (8,5.5) {};   
      \node[shape=rectangle, draw=black, fill = black, scale = .5]  at (5,5.5) {};

    \node at (1,7.5) {$y$};
    \node at (7,7.5) {$z$};

\end{tikzpicture}
\caption{Illustration of the event in~\protect\eqref{expmoma1}. The dashed line represents a time $s$ up to which
both processes $X^y$ and $X^z$ have cardinality at least $1$ and after which the process $X^y$ has cardinality at least $2$. }
\label{fig12}
\end{figure}

Using the independence of $X^y,X^z$, and 
   the union bound one has  the estimate
\begin{align}  \label{expmoma1}
&\bbP(|F^{(t-1)}_1|\geq 2) 
\\
& \quad \leq
\sum_{y,z\in[n]}\sum_{s=0}^{t-1}
 \bbP\bigl(|X^y_\ell|\geq 1 \,,\;\forall \,0\leq \ell\leq s\,, |X^y_\ell|\geq 2\,,\;\forall s< \ell< t\bigr)\,\bbP\bigl(|X^z_\ell|\geq 1\,,\;\forall \,0\leq \ell\leq s\bigr).
\nonumber
\end{align}
Letting $N^y(t)$
denote the size of the whole population of $X^y$ up to time $t-1$, one has 
\[
N^y(t)=\sum_{\ell=1}^t|X^y_{\ell-1}|=1+\sum_{\ell=1}^{t-1}|X^y_{\ell}|.
\] 
We note that for any $s\geq 0$, the event $\{|X^y_\ell|\geq 1\,,\;\forall\,0\leq \ell\leq s\}$ implies $N^y(s+1)\geq s+1$, and the events $\{|X^y_\ell|\geq 1\,,\;\forall \,0\leq \ell\leq s\}$ and $\{|X^y_\ell|\geq 2\,,\;\forall s< \ell< t\}$ together imply \[N^y(t)\geq 2(t-s-1) + s+1 = 2t - s-1.\]  Therefore, for any fixed $a\in(0,1)$, from Lemma~\ref{lem:expmom} and Markov's inequality we obtain
\begin{align}\label{expmoma2}
\bbP(|X^y_\ell|\geq 1\,,\;\forall \ell\leq  s\,, |X^y_\ell|\geq 2\,,\;\forall s< \ell< t)\,\bbP(|X^z_\ell|\geq 1\,,\;\forall \ell\leq s)\leq C_a^2 \,2^{-2at}.
\end{align}
  In conclusion, we deduce that
  \begin{align}\label{expmoma3}
\bbP(|F^{(t-1)}_1|\geq 2)  \leq 
n^2 t \,C_a^2 \,2^{-2at}.
\end{align}
Since $a$ can be taken arbitrarily close to $1$, this proves the desired estimate \eqref{dkz1} by taking $t2^{-2at}\leq (4-\d)^{-t}$ for all $t$ large enough and adjusting the value of the constant~$C_\delta$ in order to cover all values of $t\in\bbN$.  This concludes the proof of Lemma~\ref{fraglemma}.\qed

Finally, it remains only to prove Lemma \ref{lem:expmom}. 

\subsection{Proof of Lemma \ref{lem:expmom}}
We note that $a$ cannot be taken larger than $1$ since at each step there is a probability at least $1/2$ of staying alive, so that the event $N^y(t)\geq t$ has probability at least $2^{-(t-1)}$.  We start by considering the distribution of the component $G(x)$, namely
 \begin{align}\label{offspringax}
 \nu_x (U) =  \sum_{G\in\cG}\nu_{\bf J}(G)\,\ind_{G(x)=U}\,,
 \end{align}
 and establish an exponential tail bound on its size. To this end,  
 define
  \begin{align}\label{rhobar}
\r_0:=\max_x\sum_{y\neq x}p_{xy}\,,
\end{align} 
and note from the discussion at the end of Section~\ref{sec:coupling} that $\r_0\leq 4\d_0$, where $\d_0$ is the constant in the Dobrushin condition~\eqref{Dcondition}.
Now observe that if $\r_x$ is 
given by \eqref{offspringa}, then by the union bound
  \begin{align}\label{rhozero}
\r_x=1-\nu_x (\emptyset)\leq  \sum_{y\neq x}p_{xy} \leq \r_0\,.
\end{align}
 Let $\cU_x$ denote the random variable with distribution $\nu_x$.
Note that either $|\cU_x|=0$ or $|\cU_x|\geq 2$. We are going to prove a tight tail estimate for $|\cU_x|$. 
\begin{lemma}\label{lem:tailux}
Suppose $\r_0<1/4$. For every $x\in[n]$,  the size of the connected component $\cU_x$ 
satisfies \begin{align}\label{emom1}
\bbP(|\cU_x|\geq \ell) \leq  2(4\r_0)^{\ell-1}
\,,\qquad \ell\geq 2. 
\end{align}
\end{lemma}
\begin{proof}
We shall argue that the random variable $\cU_x$ can be stochastically dominated by the total population of a labeled branching process where at each step an individual $x$ gives birth to a set $S$ distributed according to the neighborhood of $x$. More formally, let 
$\cV_x$ denote the set of neighbors of $x$ in $G$, i.e., $\cV_x=\{y\in V_G:\, \{x,y\}\in E_G\}$ and denote by~$\g_x$  the distribution of $ \cV_x$: thus, for any $S\subseteq [n]\setminus\{x\}$, 
\begin{align}\label{offspringbx}
 \g_x (S) =  \prod_{y\in S}p_{xy}\prod_{z\in ([n]\setminus\{x\})\setminus S}(1-p_{xz}).
 \end{align}
 Consider the labeled branching process $W^x$ with initial value $W^x_0=\{x\}$ and offspring distribution $\g_x$. 
A breadth-first search starting at $x$ then shows that $\cU_x$ is stochastically dominated (in the sense of inclusion) by the total population of $W^x$, i.e., the union of all sets $W^x_\ell$, $\ell=0,1,\dots$, where $W^x_\ell$ represents the $\ell$-th generation of $W^x$. 
To control the size of  the total population of $W^x$, we note that 
for any integer $\ell\geq 1$, a union bound and the multinomial theorem show that  
\begin{align}\label{rhoxz}
\nu_{\bf J}(|\cV_x| \geq \ell)&=\sum_{S:\,|S|\geq \ell}\g_x(S)\leq \sum_{S:\,|S|=\ell}\prod_{y\in S}p_{xy}
\nonumber \\ &  \leq
\sum_{\{n_y\}:\sum_y n_y= \ell} \;\prod_y p_{xy}^{n_y}= \Bigl(\sum_{y\neq x}p_{xy}\Bigr)^\ell\leq\r_0^\ell
\,,
\end{align} 
where the third sum extends over all nonnegative integer vectors $\{n_y, y\in [n]\setminus\{x\}\}$ such that $\sum_{y}n_y = \ell$. 
We use \eqref{rhoxz} to define a random variable $\cN$ such that 
the size $|\cV_x|$ of the neighborhood at $x$ is stochastically dominated by $\cN$, and
$\cN$ has  distribution $\bar \mu$ given by 
\begin{align}\label{mubar}
\bar\mu(0)= 1-\r_0\,,\qquad  \bar\mu(\ell)= \r_0^ {\ell}(1-\r_0)\,,\;\;\ell\geq 1\,.
\end{align}
Indeed, provided $\r_0<1$, \eqref{mubar} defines  the geometric distribution with parameter $1-\r_0$, and by  \eqref{rhoxz}  one has
\begin{align}\label{rhoxza}
\nu_{\bf J}(|\cV_x| \geq \ell)
\leq \sum_{k\geq \ell}\bar\mu(k) = \r_0^\ell\,,
\end{align} 
for all $\ell\geq 0$. This implies the desired stochastic domination, so that $\cV_x$ and $\cN$ can be coupled so that  $|\cV_x|\leq \cN$. Thus, the size of the total population of $W^x$ is dominated by the total size $|\cT|$ of the Galton Watson tree $\cT$ with offspring distribution $\bar\mu$. In particular, we can use a coupling of $\cU_x$ and $\cT$ such that $|\cU_x|\leq |\cT|$.

Let $\varphi(s)=\bbE[s^\cN]$ denote the moment generating function of $\cN$. Standard calculations 
show that the extinction probability for $\cT$ is given by the smallest positive solution $s_{\rm ext}$ of $s=\varphi(s)$, and since $\bbE[\cN]=\r_0/(1-\r_0)<1$ 
here one has $s_{\rm ext}=1$. Recall also that $\varphi$ is increasing and convex. With similar calculations one finds that for any fixed $u>0$, the exponential moment 
\begin{align}\label{expom}
\bbE\left[u^{|\cT|}\right] = s_*(u)\,,
\end{align}
where $s_*(u)$ is defined as the smallest positive solution of the equation $s=u\varphi(s)$. 
Note that we are free to choose~$u$.
Let us show that, if we take $u=1/(4\r_0)$, then $s_*(u)\leq 1/(2\r_0)$. Indeed, by explicit calculation, 
\[
\varphi(s) = \frac{1-\r_0}{1-s\r_0}\,, 
\]
and $\varphi(s)\leq (1-\r_0)(1+2s\r_0)\leq 1+2s\r_0$ if $s\r_0\leq 1/2$. Therefore,
\[
u\varphi(s)=\frac{\varphi(s)}{4\r_0}\leq \varphi_0(s):= \frac1{4\r_0}+\frac{s}2.
\]
Noting that the solution of the equation $s=\varphi_0(s)$ is $s_0=1/(2\r_0)$, we see that the desired solution $s_*(u)$ of $s=u\varphi(s)$ exists and satisfies $ s_*(u)\leq 1/(2\r_0)$.

From \eqref{expmom} and the previous arguments, using Markov's inequality one finds
\begin{align}\label{expomt}
\bbP(|\cU_x|\geq \ell) \leq\bbP(|\cT|\geq \ell)
\leq (4\r_0)^\ell \bbE\left[(4\r_0)^{-|\cT|}\right] \leq 2(4\r_0)^{\ell-1}\,, \qquad \ell\geq2\,.
\end{align}
This proves \eqref{emom1}.
\end{proof}

We now turn to the proof of Lemma \ref{lem:expmom}. Recall the definition of $X^y$ as the labeled branching process with offspring distribution $\mu_x$ from \eqref{offspringx}. Call $\tilde \cU_x$ the random variable with distribution $\mu_x$. Let $\tilde \nu$ be defined as 
\begin{align}\label{nubar}
\tilde \nu(0)= \tfrac12-\r_*\,,\quad \tilde \nu(1)= \tfrac12\,,\quad \tilde \nu(\ell)= 2(1-4\r_0)(4\r_0)^ {\ell-1}\,,\;\;\ell\geq 2\,,\end{align}
where $\r_*=8\r_0$. Note that this is a well defined probability provided $16\r_0\leq 1$.
Let $\tilde \cN$ denote the integer valued random variable with distribution $\tilde \nu$. 
Lemma \ref{lem:tailux} implies that $|\tilde \cU_x|$ is stochastically dominated by $\tilde \cN$. To see this, note that 
for any $k\geq 2$, the event $|\tilde \cU_x|=k$
has the same probability as the event $|\cU_x|=k$, where $\cU_x$ is the random variable with distribution~$\nu_x$. 
From Lemma~\ref{lem:tailux} we know that 
\begin{align}\label{expom3}
\bbP(|\tilde \cU_x|\geq \ell)=\bbP(|\cU_x|\geq \ell) \leq
2(4\r_0)^{\ell-1}=\sum_{k\geq \ell}\tilde \nu(\ell)\,, \qquad \ell\geq2\,.
\end{align}
Moreover, by the definition of $\mu_x$ one has 
$\bbP(|\tilde \cU_x|=1)= \tfrac12-\r_x\leq \tfrac12$, and therefore
\begin{align}\label{expom3a}
\bbP(|\tilde \cU_x|\geq \ell)\leq
\sum_{k\geq \ell}\tilde \nu(\ell)\,, \qquad \ell\geq0\,.
\end{align}
This proves the stochastic domination $|\tilde \cU_x|\leq \tilde \cN$. 

Recall that 
$N^y(t)$
denotes the size of the whole population of the labeled branching process $X^y$ up to time $t-1$. The above domination argument implies that, for any $t\in\bbN$,  we can dominate $N^y(t)$ by the size $|\tilde \cT|$ of the total population of the Galton Watson tree~$\tilde\cT$ with offspring distribution~$\tilde \nu$, so that 
\begin{align}\label{exapom}
\bbE\left[2^{aN^y(t)}\right]\leq \bbE\left[2^{a|\tilde \cT|}\right]\,.
\end{align}
As in \eqref{expom} we have that, for any fixed $u>0$, 
$\bbE\left[u^{|\tilde \cT|}\right]=\tilde s_*(u)$, where $\tilde s_*(u)$ is the smallest solution of the equation $s=u\tilde\varphi(s)$, for
\[\tilde\varphi(s)=\bbE\left[s^{\tilde\cN}\right],\]
the generating function of $\tilde \nu$. 
We calculate
 \begin{align}\label{expmomas}
\tilde\varphi(s) =\frac12-\r_* + \frac{s}2 +2(1-4\r_0) \sum_{\ell\geq 2}
(4\r_0)^ {\ell-1}
s^\ell \leq \frac12(1+s+s\,\k(s\r_0)),
\end{align}
where $\k(t) = 16t/(1-4t)$. 
Set $u=2(1-\eta)=2^a$ for some fixed positive $\eta\leq 1/2$, and note that $\k(t)\leq \eta$ if $t\leq \eta/18$ for all $\eta\in[0,1/2]$. Thus, assuming
$s\r_0\leq \eta/18$, one has 
\[u\varphi(s) \leq (1-\eta)(1+ s(1+\eta))\leq (1-\eta)+s(1-\eta^2).\] Therefore, 
$\tilde s_*(u)\leq (1-\eta)/\eta^2$. For this value of $s$ to also satisfy the requirement $s\r_0\leq \eta/18$, we need $\r_0\leq \eta^3/(18(1-\eta))$.  
By the discussion at the end of Section~\ref{sec:coupling}, $\d_0\leq \eta^3/(72(1-\eta))$ is sufficient for this to hold. Under this assumption one has 
 \begin{align}\label{expmomasa}
 \bbE\left[2^{a|\tilde \cT|}\right]\leq \eta^{-2}.
 \end{align} 
The estimate \eqref{expmom} then holds with $2(1-\eta)=2^a$, $C_a=\eta^{-2}$, provided $\d_0 \leq  \eta^3/(72(1-\eta))$. This completes
the proof of Lemma \ref{lem:expmom}. \hfill$\square$

\begin{remark}\label{rem:opte0} 
Note that the value of~$\d_0$ required for Lemma~\ref{lem:expmom} to hold is precisely the value of~$\d_0$ that we will require in the
high-temperature condition~\eqref{Dcondition} in our main result for the nonlinear block dynamics, Theorem~\ref{th:mainthla}.  We have made no
attempt here to optimize the dependence of~$\d_0$ on our arguments, and with further work one could no doubt obtain a better,
explicit bound.  However, it is clear that our  current arguments will not be able to obtain an optimal value of~$\d_0$, matching Dobrushin-type
thresholds for rapid mixing for linear Markov chains.  For instance, in the mean field case where ${\bf J} \equiv \b/n$, the threshold
for the equilibrium phase transition is $\b=1$, and thus $\d_0<1$ would be an optimal condition.
\end{remark}

\section{Nonlinear Glauber dynamics}\label{sec:Glauber}
Recalling the definition~\eqref{pq22}, the  interaction \eqref{consepq} in the nonlinear Glauber dynamics takes the form 
\begin{equation}
\label{cw2ss}
p\circ q= \frac1n\sum_{x\in[n]}\sum_{\si,\si'}\tfrac12(p(\si)q(\si')+q(\si)p(\si')) Q_x(\cdot\mid \si,\si'),
\end{equation}
where $Q_x(\t\mid \si,\si') := \sum_{\t'}\cQ_{{\bf J},x}(\si,\si';\t,\t')$.
We write $S_t(p)$, $t\in\bbN$, for the $t$-th iteration of $p\mapsto p\circ p$. From Theorem \ref{th:conv} we know that for any fixed interaction matrix ${\bf J}$, and any initial distribution $p\in\cP(\O)$, the above dynamics converges to~$\mu_{\bf ,h}$,
where ${\bf h}$ is the unique vector of external fields such that $ \mu_{\bf J,h}$ and the initial state $p$ have the same marginals on $\si_x$, for all $x\in V$. Our main result for the nonlinear Glauber (single site) dynamics, Theorem~\ref{thm:main} in the introduction, establishes a tight bound on the rate of convergence as a function of $n=|V|$, under the same Dobrushin-type condition~\eqref{Dcondition} on the interaction as in the case of block dynamics.  We restate this theorem here for convenience.
\begin{theorem}
\label{th:mainthx}
There exist absolute constants $\d_0>0$, $c>0$ and $C>0$ such that, if \eqref{Dcondition} holds with 
constant $\d_0$, then for any $p\in\cP(\O)$ and $t\in\bbN$,
\begin{gather}\label{convergix}
\|S_t(p)-\mu_{\bf J,h}\|_{\rm TV}\leq Cne^{-c\,t},
\end{gather}
where ${\bf h}\in\bbR^n$ is the unique choice of external fields such that $p_x=(\mu_{\bf J,h})_x$ for all $x\in [n]$. 
In particular, for any $\e>0$, one has $\|S_t(p)-\mu_{\bf J,h}\|_{\rm TV}\leq \e$ as soon as $t\geq \frac{n}{c}\log n + C_1(\e)$, where $C_1(\e)=\frac{n}c\log(C/\e)$.
\end{theorem}
The main difference with respect to Theorem \ref{th:mainthla} is the rate of exponential decay, which is $n$ times slower in this case. This reflects the intuitive fact that only one spin is exchanged at each time step whereas $\Theta(n)$ spins are exchanged in a typical transition of the nonlinear block dynamics studied in Theorem~\ref{th:mainthla}.

The proof of Theorem \ref{th:mainthx} follows a similar strategy to that of Theorem~\ref{th:mainthla}, but with some important technical differences. We begin with an analog of Lemma~\ref{lem:decca}, in which the role of Erd\H{o}s-R\'enyi graphs is now played by random star graphs centered at a specific vertex~$x$.

\subsection{Coupling 
with inhomogeneous random star graphs} 
We start with a single site version of Lemma~\ref{lem:gamuJ}.
\begin{lemma}\label{lem:gamuJx}
For any $x\in[n]$, $\si,\si'\in\O$, 
\begin{gather}\label{gammalao}
\a_x(\si,\si')=\frac1{1+
\exp\left\{(\si_x-\si'_x)\sum_{y:\,y\neq x}  J_{xy}(\si_y-\si'_y) \right\}}\,.
\end{gather}
\end{lemma}
\begin{proof}
The expression \eqref{gammalao} follows directly from~\eqref{eq:Gibbs}, the definition \eqref{pq22} and the observation that
\begin{align*}
\sum_{y,z}J_{y,z}\big[(\si_y\si_z + \si'_y\si'_z)-& ([\si_x\si'_{[n]\setminus \{x\}}]_y[\si_x\si'_{[n]\setminus \{x\}}]_z +[\si'_x\si_{[n]\setminus \{x\}}]_y[\si'_x\si_{[n]\setminus \{x\}}]_z)\big] \nonumber \\&\qquad\qquad \qquad = 
(\si_x-\si'_x)\sum_{y:\,y\neq x}  J_{xy}(\si_y-\si'_y)\,. \qedhere
\end{align*}
\end{proof}
\begin{remark}
\label{rem:alpx}
Consider the random variable $\eta_x$ which takes the value $-1$ with probability $\a_x(\si,\si')$, and $+1$ with probability $1-\a_x(\si,\si')$. Then by Lemma \ref{lem:gamuJx}, its probability density can be written as $\a_x(\cdot\tc\si,\si')$, where  
\begin{gather}\label{gammalaoa}
\a_x(\eta_x\tc\si,\si')  \propto
\exp\Bigl\{\eta_x\sum_{y:\,y\neq x} \tilde J_{xy} \Bigr\}\,,\qquad \eta_x\in\{-1,+1\}\,,
\end{gather}
and $\tilde J_{xy}$ is defined as in \eqref{jtilde}. This also shows that $\a_x(\cdot\tc\si,\si')$
coincides with the distribution of $\eta_x$ obtained by conditioning $\g(\cdot\tc\si,\si')$ on the event that $\{\eta_y=1,\;\forall \,y\neq x\}$, where $\g(\cdot\tc\si,\si')$ is the distribution from 
Lemma \ref{lem:gamuJ}. Note also that, if $\si_x=\si'_x$, then $\tilde J_{xy}= 0$ for all $y\neq x$, and therefore $\a_x(\cdot\tc\si,\si') = {\rm Be}(1/2)$ is the Bernoulli probability measure on $\{-1,+1\}$ with parameter $1/2$.
\end{remark}

Fix $x\in[n]$, and let $\cG_x$ be the set of all subgraphs of the star graph $S_x=(V,E_x)$ with vertex set $V=[n]$, and edge set $E_x:=\{\{x,y\},\;y\in[n]\setminus\{x\}\}$. We view $G\in\cG_x$ as a collection of edges, i.e., a subset of $E_x$,  with no isolated vertices, and write $\cP(\cG_x)$ for the set of  probability measures over $\cG_x$.  Note that any $G\in\cG_x$ is always connected, but can be the empty graph (with no vertices). We use $V_G,E_G$ for the vertex and edge sets of $G\in\cG_x$, respectively.    The following is the single site version of Lemma \ref{lem:decca}.
\begin{lemma}\label{lem:deccax}
Fix $x\in [n]$. Let $\nu^x_{\bf J}$ be the inhomogeneous random star measure associated with the weights $\l_{xy}=e^{4|J_{xy}|}-1$, i.e., 
\begin{gather}\label{aformappx}
\nu^x_{\bf J}(G) \propto \prod_{\{x,y\}\in E_G} (e^{4|J_{xy}|}-1)\,,\qquad G\in\cG_x\,.
\end{gather}
Then
\begin{align}
\label{nug1x}
\a_x(\cdot \tc\si,\si') = \sum_{G\in\cG_x}\nu^x_{\bf J}(G)\, \mu^x_{G}(\cdot\tc \si_{V_G},\si'_{V_G})
\,,
\end{align}
where, for any $G\in\cG_x$, $\mu^x_{G}(\cdot \tc \si_{V_G},\si'_{V_G})$ is a probability measure on $\{-1,+1\}$ that depends on $\si,\si'$ only through the spins $\si_{V_G},\si'_{V_G}$ and such that if $G=\emptyset$, then $\mu^x_{G}(\cdot \tc \si_{V_G},\si'_{V_G})={\rm Be}(\tfrac12)$ is the Bernoulli probability measure on $\{-1,+1\}$, with parameter $1/2$. 
\end{lemma}
\begin{proof}
The proof is similar to that of Lemma  \ref{lem:decca} and is omitted.
\end{proof}

\subsection{Coupon collecting with noise}
In the case of single site updates, fragmentation involves hitting each site at least once.  The analog of fragmentation plus
noise can therefore be viewed as a coupon collecting process, again perturbed by a local growth described by simpler star graphs.

Given a set $A\subseteq [n]$, a uniformly random vertex $x\in[n]$, a Bernoulli variable $B\in\{-1,+1\}$ with parameter $1/2$ and a random graph $G\in\cG_x$ with distribution $\nu^x_{\bf J}$, we define the random sets 
  \begin{align}\label{psi01}
\Psi_0(A) = \Psi_1(A) = \emptyset \,,\qquad \text{if\;\;} |A|\leq 1\,,
\end{align}
and, if $|A|\geq 2$, 
  \begin{align}\label{psi011}
(\Psi_0(A),\Psi_1(A)) = \begin{cases}
(A,\emptyset) &\text{if\;}  x\notin A\\
(A,\emptyset) &\text{if\;}  x\in A\,,\; G=\emptyset\;\text{and\;} B=-1\\(A\setminus\{x\},\emptyset) &\text{if\;}  x\in A\,,\;G=\emptyset\;\text{and\;} B=+1\\
(A\cup V_G,V_G) & \text{if\;} x\in A\; \text{and\;} G\neq\emptyset
\end{cases}
\end{align}

\begin{definition}\label{def:fragnoisex}
The {\em coupon collecting plus noise} process $\cC_t$, $t=0,1,\dots$ is the random process defined as follows. For each $t\in\bbN$, $\cC_t$ consists of $2^t$ labeled fragments, i.e., (possibly empty) subsets $C^{(t)}_1,\dots,C^{(t)}_{2^t}$, $C^{(t)}_i\subseteq [n]$, obtained by repeated application of the following rule.  
At time zero we have $\cC_0=\{[n]\}$, i.e., $C^{(0)}_1=[n]$. At time $t\in\bbN$, 
if $\cC_{t-1}=(C^{(t-1)}_1,\dots,C^{(t-1)}_{2^{t-1}})$, then for each $i$ independently, the fragment $C^{(t-1)}_i$ is replaced by $(\Psi_0(C^{(t-1)}_i),\Psi_1(C^{(t-1)}_i))$ where $\Psi_0,\Psi_1$ are the random maps defined by \eqref{psi01}-\eqref{psi011}, so that 
\begin{align}\label{phragx}
\cC_{t}&=(C^{(t)}_1,\dots,C^{(t)}_{2^t})\nonumber \\&=(\Psi_0(C^{(t-1)}_1),\Psi_1(C^{(t-1)}_1),\Psi_0(C^{(t-1)}_2),\Psi_1(C^{(t-1)}_2),\dots,\Psi_0(C^{(t-1)}_{2^{t-1}}),\Psi_1(C^{(t-1)}_{2^{t-1}}))\,. 
\end{align}
The process is said to\/~{\rm die out} if there is a time $t$ such that $C^{(t)}_i=\emptyset$ for all $i=1,\dots,2^t$. In the latter case, with slight abuse of notation, we write $\cC_t=\emptyset$. 
\end{definition}
\begin{remark}\label{rem:coupx}
In the non-interacting case ${\bf J}=0$ one has always $G=\emptyset$
and thus $\Psi_1(A)=\emptyset$, so that the process $\cC_t$ is a monotone decreasing sequence obtained by iterations of the map $A\mapsto \Phi_0(A)$, where at each step one picks a uniformly random vertex $x\in[n]$ and if $x\in A$ the vertex $x$ is removed from $A$ with probability $1/2$. This yields a lazy version of the standard coupon collecting process, and therefore the union bound shows that
\begin{align}\label{dkzerox}
\bbP(\cC_t\neq\emptyset)\leq n\,\left(1-\tfrac1{2n}\right)^{t}\leq n \,e^{-t/(2n)}\,,\qquad t=1,2,\dots
\end{align}
In the interacting case, the event $G\neq\emptyset$ causes the splitting of $A$ into two nontrivial sets $(\Psi_0(A),\Psi_1(A))$. This local growth can slow down the convergence to the  absorbing state $\emptyset$. However, we will show that if the growth is sufficiently subcritical (ensured by the Dobrushin condition~\eqref{Dcondition}) then a bound of the form \eqref{dkzerox} continues to hold.  
\end{remark}

Despite the formal similarities between Definition \ref{def:fragnoisex} and Definition \ref{def:fragnoise}, the two processes are quite different. Note in particular the strong asymmetry between the different fragments in the case of Definition \ref{def:fragnoisex}.
Let us first consider the evolution of the leftmost fragment $C^{(t)}_1$, $t=0,1,\dots$.

\begin{lemma}\label{fraglemma1x}
There exist absolute constants $\d_0>0$, $a>0$ and $C>0$ such that if  
\eqref{Dcondition} holds with constant~$\d_0$ then
\begin{align}\label{dkz1x}
\bbP(C^{(t)}_1\neq\emptyset)\leq Cne^{-at/n}\,,\qquad t=1,2,\dots
\end{align}
\end{lemma}
\begin{proof}
The  leftmost fragment $C^{(t)}_1=\Phi_0\circ\cdots\circ\Phi_0([n])$ is obtained by iterating $t$ times the map $\Phi_0$. It may be represented by the following 
labeled branching process.
Consider random variables $W:=\{W_t,\,t=0,1,\dots\}$, where $W_0=[n]$ and for each $t$, $W_t\subseteq [n]$ is updated to $W_{t+1}\subseteq [n]$ by the rule
\begin{itemize}
\item pick $x\in[n]$ u.a.r.
\item if  $x\notin W_t$ then set $W_{t+1}=W_t$
\item if  $x\in W_t$ then let $\cV_x$ denote the neighborhood of $x$ in $G$, i.e., the random variable with distribution as in~\eqref{offspringbx}: if $\cV_x\neq\emptyset$, set $W_{t+1}=W_t\cup \cV_x$; if instead $\cV_x=\emptyset$, set 
\[
W_{t+1}=\begin{cases}W_t\setminus\{x\}
& \text{with prob.\;}\tfrac12\\
W_t
& \text{with prob.\;}\tfrac12
\end{cases}
\]
\end{itemize}
From our definitions we see that $C^{(t)}_1$ has the same distribution as~$W_t$.

Next, consider a continuous time version of the process $W_t$. Namely, let $T_1,T_2,\dots$ denote the arrival times of the Poisson point process with intensity  $1$ on $[0,\infty)$
and call $\widetilde W_t$, $t\geq 0$, the process obtained by repeating the steps above at each arrival time $T_j$, so that $W_t = \widetilde W_{T_t}$ for every $t\in\bbN$. 

From standard properties of the Poisson process, the random process $\widetilde W_t$ can be equivalently obtained by attaching to each $x\in [n]$ an independent Poisson process with intensity $1/n$, and performing updates at~$x$ independently at the arrival times of the process at~$x$.  Neglecting the overlap between different branches, and using the domination argument in the proof of Lemma \ref{lem:tailux}, one can stochastically dominate the cardinality of $\widetilde W_t$ by
\begin{align}\label{dkz2x}
|\widetilde W_t|\leq \sum_{x\in[n]}M^x_t,
\end{align}
where $M^x_t$, $x\in[n]$ are i.i.d.\ copies of the continuous time branching process obtained by setting $M^x_0=1$ and then letting each branch evolve independently with branching times given by exponential random variables with parameter $1/n$.  In this process, each individual gives birth to $\ell$ individuals with probability $\mu_*(\ell)$ given 
by
\begin{align}\label{eq:mustar}
\mu_*(\ell)=\begin{cases}\tfrac12\bar\mu(0)
& \ell=0\\
\tfrac12\bar\mu(1)
& \ell=1\\
\bar\mu(\ell)
& \ell\geq 2
\end{cases}
\end{align}
where $\bar\mu$ is defined as in~\eqref{mubar} with the same value of~$\r_0$. 
(See~\cite{athreya1972branching} for background on continuous time branching processes.)
Since $M^1_t$ is subcritical as soon as $\d_0$ is sufficiently small, and $\bar\mu$ has exponential tails, standard
estimates on subcritical continuous time branching processes
imply that  $\bbP(M^1_t>0)\leq Ae^{-at/n}$ for some constants $a,A>0$. 
It follows that 
\begin{align}\label{dkz3x}
\bbP(|\widetilde W_t|>0) \leq n\,\bbP(M^1_t>0)\leq nAe^{-at/n}\,.
\end{align}
To conclude the proof, recall  that $W_t = \widetilde W_{T_t}$ and that $T_t$ is the sum of $t$ i.i.d.\ exponentials of parameter $1$, so that 
\begin{align}\label{dkz4x}
\bbP(T_t\leq t/2) \leq e^{-c t}\,,
\end{align}
for some constant $c>0$. Therefore, for all $t\in\bbN$,
\begin{align}\label{dkz5x}
\bbP(C^{(t)}_1\neq\emptyset)\leq \bbP(|\widetilde W_{T_t}|>0)
\leq \bbP(|\widetilde W_{t/2}|>0) + e^{-c t}\leq n\,Ae^{-at/(2n)}  +e^{-c t} \,,
\end{align}
completing the proof of the lemma.
 \end{proof}%
Next, we extend the idea of the proof of Lemma \ref{fraglemma1x} to obtain the following stronger estimate.

 \begin{lemma}\label{fraglemma2x}
There exist absolute constants $\d_0>0$, $a>0$ and $C>0$ such that if  
\eqref{Dcondition} holds with constant~$\d_0$ then
\begin{align}\label{dkz1xx}
\bbP(\cC_t\neq\emptyset)\leq Cne^{-at/n}\,,\qquad t=1,2,\dots
\end{align}
\end{lemma}
\begin{proof}
The idea is to add all individuals from the other fragments to the evolution of the first fragment, which we interpret as the ``master evolution"
to which all further fragments are added, and to dominate this global process by a single sub-critical branching process.
Indeed, notice that every time there is a branching event in the 
coupon collecting plus noise process, at a fragment $C_i^{(j)}$, say at vertex $x$, if a nonempty neighborhood $\cV_x$ of $x$ is drawn, then the set $\cV_x$ is added to both $\Phi_0(C_i^{(j)})$ and $\Phi_1(C_i^{(j)})$.  Thus, recursively, we can add all of them to the master evolution and obtain a global process $\widetilde\cC = \{\cC_t, \,t=0,1,\dots\}$ that evolves like the first fragment, with the rule that every time a $\cV_x\neq\emptyset$ is drawn then {\it two\/} copies of that set are added to the current configuration, each of which continues its evolution independently according to the same branching rules as above. In particular, the same argument as in the proof of Lemma \ref{fraglemma1x} shows that 
\begin{align}\label{dkz7x}
\bbP(\cC_t\neq\emptyset)\leq \bbP(Z_{T_t}\neq\emptyset)\,,
\end{align}
where $Z_t$, $t\geq 0$, is the continuous time process defined by $Z_{t}=\sum_{i=1}^nZ^i_t$, where $Z^i_t$ are i.i.d.\ copies of a process $Z^1_t$ defined as follows. 
At time $0$ there is one individual, and each individual that is alive at any time $s\geq 0$, independently of all others, waits an exponential time with parameter $1/n$ and then branches into $0,1$, or $2j$ individuals, with probabilities given, respectively by, $\mu_*(0),\mu_*(1)$, and $\mu_*(j)$, $j\geq 2$.    If the branching produces $2j$ individuals, then these are divided into two groups $(y_1,\dots,y_j)$ and $(y'_1,\dots,y'_j)$, which continue their evolutions independently except that the branching times of the two groups are coupled so that they are identical. (This slightly strange coupling is a consequence of the discrete nature of time in the original process $\cC_t$.) 

Let us first consider the simpler situation where we do not have the coupling of branching times every time a  branching produces $2j$ individuals. 
Indeed, let us call $\widetilde Z^1_t$ the process evolving as above with the feature that the two groups  $(y_1,\dots,y_j)$ and $(y'_1,\dots,y'_j)$ evolve in a fully independent way following the above described rules. Thus, $\widetilde Z^1_t$ is a pure continuous-time branching process with offspring distribution
\begin{align}\label{eq:nustar}
\nu_*(\ell)=\begin{cases}\tfrac12\bar\mu(0)
& \ell=0\\
\tfrac12\bar\mu(1)
& \ell=1\\
\bar\mu(\ell/2)
& \ell=4,6,\dots
\end{cases}
\end{align}
The process $\widetilde Z^1_t$ is subcritical as soon as $\d_0$ is sufficiently small, and since $\bar\mu$ has exponential tails,  
for some constants $a,A>0$ one has
\begin{align}\label{dkz81x}
\bbP(\widetilde Z^1_t>0)\leq Ae^{-at/n}\,.
\end{align}
In particular, letting $\widetilde Z_{t}=\sum_{i=1}^n\widetilde Z^i_t$, we conclude, as in the proof of Lemma \ref{fraglemma1x}, that 
\begin{align}\label{dkz8x}
\bbP(\widetilde Z_{T_t}\neq\emptyset)\leq 
n\,Ae^{-at/(2n)} +
e^{-c t} \,.
\end{align}
Thus the proof would be complete if we could replace $Z_{T_t}$ with $\widetilde Z_{T_t}$
in \eqref{dkz7x}. In order to finish the proof it is then sufficient to show that
\begin{align}\label{dkz72x}
\bbP(\cC_t\neq\emptyset)\leq \bbP(\widetilde Z_{T_t}\neq\emptyset) + e^{-c t}\,,
\end{align}
for some constant $c>0$.

To prove \eqref{dkz72x}, recall the definition of the random process $\widetilde W_t$ obtained by attaching  independent Poisson processes with intensity $1/n$ to each $x\in [n]$, such that $C_1^{(t)}= \widetilde W_{T_t}$. Let us add to the master evolution $C_1^{(t)}$ all other evolutions as explained above by doubling the splittings every time there is a branching involving a nonempty neighborhood of $x$. This yields a global process $\cQ_t$, $t\geq 0$ on a single evolution, such that $\{\cC_t=\emptyset\} = \{\cQ_{T_t}=\emptyset\}$. Note that the branching times (but not the choices of vertices $x$ to be updated) of newly added evolutions are fully coupled to the branching times of the master evolution dictated by the original Poisson process. If, for every newly opened evolution we use instead an independent realization of the Poisson point process, then we obtain a global process $\widetilde Q_t$, which is dominated as explained above by the process $\widetilde Z_t$. Thus it is sufficient to find a coupling of $\cQ_t$ and $\widetilde \cQ_t$ such that 
\begin{align}\label{dkz9ax}
\bbP(\cQ_{T_t}\neq\emptyset)\leq \bbP(\widetilde\cQ_{T_t}\neq\emptyset) + e^{-bt}\,,
\end{align}
 for some constant $b>0$. 

Note that if $b>0$ is a sufficiently small constant, then a simple large deviation bound shows that the probability that a Poisson process of intensity $1$ has a total number of arrivals in a time $t/2$ that is less than $b t$,  is at most $3^{-t}$. Thus, if there are $2^t$ such processes, the event $\cE_t$ that all of them have at least $bt$ arrivals within time $t/2$ satisfies 
\begin{align}\label{dkza8x}
\bbP(\cE_t^c)\leq 2^t3^{-t} = (2/3)^t. 
\end{align}
We can couple $\cQ_{t},\widetilde \cQ_{t}$ in such a way that, on the event $\{T_t\geq t/2\}\cap \cE_t$, we have $ \{\cQ_{T_t}\neq\emptyset\}\subset\{\cQ_{T_{bt}}\neq\emptyset\}$.  
It follows that, 
 \begin{align}\label{dkze9ax}
\bbP(\cC_t\neq\emptyset)=\bbP(\cQ_{T_t}\neq\emptyset)&\leq \bbP(\cQ_{T_t}\neq\emptyset,\,T_t\geq t/2, \cE_t) + (2/3)^t+ e^{-ct}\nonumber\\
& \leq \bbP(\widetilde\cQ_{T_{bt}}\neq\emptyset) + (2/3)^t+e^{-ct}\,.
\end{align}
On the other hand, 
$\widetilde\cQ_{t}\neq\emptyset$ implies $\widetilde Z_t\neq 0$, and thus, by \eqref{dkz8x},
 \begin{align}\label{dkze9ax1}
\bbP(\widetilde\cQ_{T_{bt}}\neq\emptyset)\leq n\,Ae^{-abt/(2n)} +
e^{-bc t} \,.
\end{align}
Combining \eqref{dkze9ax} and \eqref{dkze9ax1} proves the desired estimate \eqref{dkz1xx}.
 \end{proof}

 \subsection{Proof of Theorem \ref{th:mainthx}}
As in the proof of Theorem \ref{th:mainthla}, the first step is to use 
the coupling with random (star) graphs and Bernoulli variables to decompose the distribution at every internal node of the interaction tree. For this we use Lemma \ref{lem:deccax}. This leads to the analysis of the branching process, which in this case takes the form of coupon collecting with noise. The conclusion then follows from Lemma \ref{fraglemma2x}. With the notation from Theorem \ref{th:mainarb} we are going to establish that 
there exist absolute constants $\d_0>0$, $c>0$ and $C>0$ such that, if
\eqref{Dcondition} holds with constant~$\d_0$ then, for any choice of initial distributions $\vec p,\vec q$ with the same marginals as in \eqref{samemarg}, 
\begin{gather}\label{convergiarbxx}
\|S_t(\vec p)-S_t(\vec q)\|_{\rm TV}\leq Cne^{-c\,t/n},\qquad t\in\bbN\,,
\end{gather}
where $S_t(\vec p)$ denotes the distribution at the root of the binary tree when the leaves are equipped with the distributions $\vec p$ and the single site interaction \eqref{cw2ss} occurs at each internal node.

For each interaction we draw a uniformly random vertex $x\in[n]$, and reveal a realization of the neighborhood $G=\cV_x$ and a Bernoulli variable $B\in\{-1,+1\}$. Then, 
according to Lemma \ref{lem:deccax}, the interaction~\eqref{cw2ss}
has the form 
\begin{align}\label{eq:apqt}
p\circ q= \frac1{2n}\sum_{x,G,B}\nu^x_{\bf J}(G)\tfrac12(S(p,q\tc x,G,B)+S(q,p\tc x,G,B))
\end{align}
where 
\begin{align}\label{eq:bpqt}
S(p,q\tc x,G,B)(\t)&=\sum_{\si,\si'} p(\si)q(\si')\sum_{\eta_{x}\in\{-1,+1\}}\mu^x_{G,B}(\eta_{x}\tc \si_{V_G},\si'_{V_G})
\,\ind_{\t\sim (\si,\si',\eta_{x})},
\end{align}
and, for $\t\in\O$, the notation $\t\sim (\si,\si',\eta_{x})$ stands for 
\begin{align}\label{eq:pqt2x}
\t\sim (\si,\si',\eta_{x})\;\Leftrightarrow\;\begin{cases}\t_x=\si_x\,,&\,\eta_x=+1\\
\t_x=\si'_x\,,&\,\eta_x=-1\\
\t_y=\si_y\,,&\,\forall y\neq x
\end{cases}\end{align}
and $\mu^x_{G,B}(\cdot\tc \si_{V_G},\si'_{V_G})$ is a probability distribution on $\{-1,+1\}$ that depends on $\si,\si'$ only through $\si_{V_G},\si'_{V_G}$ if $G\neq\emptyset$, while it assigns the value $\eta_x=B$ deterministically if $G=\emptyset$. This definition reflects the fact that if $G=\emptyset$ then we assign the value of~$\eta_x$ by a fair coin flip~$B$, whereas if $G\neq \emptyset$ then $B$ is irrelevant and we sample~$\eta_x$ according to $\mu^x_{G}(\cdot\tc \si_{V_G},\si'_{V_G})$.
 
We remark that our definition of the maps $\Psi_0$ and $\Psi_1$ in Definition \ref{def:fragnoisex} is such that the only information about the distributions $p,q$ that is needed to compute $S(p,q\tc x,G,B)$ is contained in the marginals $p_{\Psi_0([n])}$ and $q_{\Psi_1([n])}$. Note the asymmetry between $p,q$ in this expression: symmetry is restored in~\eqref{eq:apqt} by the averaging
$\tfrac12(S(p,q\tc x,G,B)+S(q,p\tc x,G,B))$. The choice between the two options will then be encoded by a further Bernoulli variable $B'$ with parameter $1/2$, that we sample afresh at each interaction.

Repeating the argument in the proof of  Theorem \ref{th:mainarb}, one has that after $t$ steps, conditional on the realizations $\vec x$ of all random vertices, all graphs $\vec G$ and Bernoulli variables $\vec B, \vec B'$ involved in each of the $2^t-1$ interactions, the fragment $C^{(t)}_{i}$, $i=1,\dots,2^t$, as defined in Definition \ref{def:fragnoisex}, contains all the information needed from the distribution at the leaf $\vec B'(i)$, which is determined by taking the left or right descendant from the root at each internal node according to whether $B'=+1$ or $B'=-1$ at that node. 
As before, we observe that as soon as a fragment becomes either empty or contains one site only, then the information carried by the corresponding leaf is irrelevant.
 
In analogy with \eqref{eqcq1}, we write 
\begin{gather}\label{eqcq1x}
S_t(\vec p)=\sum_{(\vec x,\vec G,\vec B,\vec B')} \widehat\nu(\vec x,\vec G,\vec B,\vec B')\, S_t(\vec p\tc \vec x,\vec G,\vec B,\vec B')
\end{gather}
with 
 \begin{gather}\label{eqcq2a1x}
 \widehat\nu(\vec x,\vec G,\vec B,\vec B')=\prod_{i=1}^{2^t-1}\tfrac1{4n}\,
 \nu^x_{\bf J}(G^{(i)})
 \end{gather}
where   $S_t(\vec p\tc \vec x,\vec G,\vec B,\vec B')$ is some probability measure that may depend on $(\vec x,\vec G,\vec B,\vec B')$ and on $\vec p=(p_1,\dots,p_{2^t})$ in a complicated way but has the property that
\begin{gather}\label{hpw}
 S_t(\vec p\tc \vec x,\vec G,\vec B,\vec B')=S_t(\vec q\tc \vec x,\vec G,\vec B,\vec B')\,,
 \qquad \text{if\;}\; (\vec x,\vec G,\vec B,\vec B')\in \{\cC_t=\emptyset\}. 
 \end{gather}
 Arguing as in \eqref{eqcq2}-\eqref{convergiarbat}, 
 it follows that 
\begin{gather}\label{convergiarbaxx}
\| S_t(\vec p)- S_t(\vec q)\|_{\rm TV}\leq \bbP(\cC_t\neq\emptyset)\,.
\end{gather}
The conclusion \eqref{convergiarbxx} now follows from  Lemma \ref{fraglemma2x}.
This ends the proof of Theorem \ref{th:mainthx}.  \hfill $\square$

\section{Concluding remarks and open questions}\label{sec:open}
Our results leave open a number of interesting directions for further research, some of which we briefly mention here, 
along with some additional observations and extensions.

\begin{enumerate}
\item Can we prove exponential decay (even with an exponentially bad dependence on~$n$) for the nonlinear Ising dynamics for 
{\it arbitrary\/} interactions~${\bf J}$, without the high-temperature condition?  Recall that we did prove convergence for arbitrary~$\bf J$ in Theorem~\ref{thm:convintro},
but unlike its analogs for linear Markov chains our proof apparently gives no useful rate information (see Remark \ref{rem:irr}).

\item Can our high-temperature condition $\max_x\sum_{y\in V} |J_{xy}|\leq \d_0$ be relaxed to accommodate an optimal value
for~$\d_0$, in particular in the case of the complete graph (mean-field or Curie-Weiss model)?  Here all interactions $J_{xy}=\frac{\beta}{n}$,
and $\beta=1$ marks the phase transition.  Thus one might hope to sharpen our results to require only $\d_0<1$ in this case.
One might also hope to replace the $\ell_1$ condition on~$\bf J$ by a spectral condition, as e.g.\ in the recent works 
\cite{bauerschmidt2024log,ChenEldan22}.

\item We have proved rapid convergence in total variation distance, a natural goal.  However, one might naturally also ask
whether one can prove sharp bounds on the rate of contraction of relative entropy, a more delicate question.  This was done
recently in the non-interacting (population genetics) case in~\cite{CapSin,CapPar}.

\item As mentioned in the introduction, our results for the Ising model generalize to any spin system with a constant number
of spins at each vertex and bounded pairwise interactions, including the $q$-state Potts model, under an analogous Dobrushin-type condition~\eqref{Dcondition}.
(In that condition, the sum is now over the maximum absolute values of all interactions involving any given site~$x$.)  This follows
from the fact that, as can readily be checked, the representation of the measure $\gamma(\cdot\mid\si,\si')$ in terms of a two-spin system
as described in Lemma~\ref{lem:gamuJ} still holds in this
more general setting, and beyond that point the rest of the analysis depends only on that representation.

\item We note that, in contrast to linear Markov chains (which can be simulated using a single point walking randomly
around the state space), nonlinear dynamics do not immediately provide an efficient
algorithm for sampling from the stationary distribution, even when the convergence time is short.  This is a consequence
of the fact that, in order to obtain a single sample from the time-$t$ distribution~$T_t(p)$, we would naively need to
simulate the entire evolutionary tree of depth~$t$, which begins with $2^t$ independent random samples from~$p$ at the leaves.
(Indeed, Arora {\it et al.}~\cite{ARV} prove that simulating arbitrary reversible quadratic dynamics is in fact {\sc Pspace}-complete.)
However, for both of the dynamics we consider here, our analysis does in fact also provide a polynomial time algorithm for sampling
from (very close to) the stationary distribution~$\mu_{\bf J,h}$.   This is immediately obvious for the nonlinear block dynamics: 
since Theorem~\ref{thm:main2} establishes that $t=O(\log n)$ steps are enough for convergence, the entire tree is of polynomial size
and thus can be constructed in polynomial time.  (A detail here is that each interaction in the tree requires sampling the
set of sites~$\Lambda$ to be exchanged according to the distribution~\eqref{jtilde}.  However, this itself is a high-temperature
Ising Gibbs measure, and thus can be sampled from efficiently by other means.)  For the nonlinear Glauber dynamics, 
a polynomial time simulation is also possible once we observe from our analysis that the actual size of the master evolution
for $S_t(p)$ is $O(t)$; since the master evolution is sufficient to construct our time-$t$ sample, this can be done in ${\rm poly}(n)$ time.

The simulations in the previous paragraph leave much to be desired algorithmically, as they are rather unwieldy and, more
significantly, do not retain the underlying pairwise interaction structure of the nonlinear dynamics.
An interesting related question, for both processes, is whether a more explicit simulation is possible: namely, starting with
a finite population of size $N={\rm poly}(n)$, each member of which is sampled independently from the initial distribution~$p$,
evolve this population in the obvious way (by carrying out interactions between randomly chosen pairs to construct
the next-generation population of the same size).  The question is whether the time-$t$ population in this finite implementation
is close to the true population~$T_t(p)$, at least for modest times~$t$---this is not obvious since unwanted correlations will arise
due to the finite population size, though these will decrease with the population size~$N$.  (This implementation, which is actually
a Markov chain on a very large state space, is known as the ``Kac model" in kinetic theory \cite{kac1956foundations}, 
and the convergence to the true population is referred to as ``propagation of chaos".)  In the non-interacting (population
genetics) case, it was shown that a low-degree polynomial population size does in fact suffice~\cite{Sinetal2,CapPar}.
It would be interesting to see if this argument can be extended to the Ising model at high temperature; in particular, 
since the key insight in~\cite{Sinetal2} comes from the fragmentation process, the question boils down to whether the same
insight carries through in the presence of noise.
\end{enumerate}

\bibliographystyle{plain}

\bibliography{nonlinear}

\end{document}